\documentclass[10pt]{book}
\usepackage{array}
\RequirePackage{amsthm,amsmath,amssymb,natbib}
\usepackage{epsf}
\usepackage{graphicx}
\usepackage{epsfig}
\usepackage{bbm}
\usepackage[active]{srcltx}
\usepackage{mathtools}
\usepackage[usenames,dvipsnames]{xcolor}

\theoremstyle{plain}
\newtheorem{thm}{Theorem}[chapter]
\newtheorem{lem}[thm]{Lemma}
\newtheorem{remark}[thm]{Remark}

\newtheorem{cor}[thm]{Corollary}

\theoremstyle{definition}
\newtheorem{defn}[thm]{Definition}

\newcommand{\ben}{\begin{enumerate}}  
\newcommand{\een}{\end{enumerate}}
\newcommand{\bed}{\begin{itemize}}  
\newcommand{\eed}{\end{itemize}}
\def\pp{\mathbb{P}}

\begin{document}


\renewcommand{\baselinestretch}{1.2}

\markright{
\hbox{\footnotesize\rm Statistica Sinica (2014): Preprint}\hfill
}

\markboth{\hfill{\footnotesize\rm Alexandra Carpentier and Arlene K. H. Kim} \hfill}
{\hfill {\footnotesize\rm Adaptive and minimax optimal estimation of the tail coefficient} \hfill}

\renewcommand{\thefootnote}{}
$\ $\par


\fontsize{10.95}{14pt plus.8pt minus .6pt}\selectfont
\vspace{0.8pc}
\centerline{\large\bf Adaptive and minimax optimal estimation}
\vspace{2pt}
\centerline{\large\bf of the tail coefficient}
\vspace{.4cm}
\centerline{Alexandra Carpentier and Arlene K. H. Kim}
\vspace{.4cm}
\centerline{\it University of Cambridge}
\vspace{.55cm}
\fontsize{9}{11.5pt plus.8pt minus .6pt}\selectfont


\begin{quotation}
\noindent {\it Abstract:}
We consider the problem of estimating the tail index $\alpha$ of a distribution satisfying a $(\alpha, \beta)$ second-order Pareto-type condition, where $\beta$ is the second-order coefficient. When $\beta$ is available, it was previously proved that $\alpha$ can be estimated with the optimal rate $n^{-\frac{\beta}{2\beta+1}}$. On the contrary, when $\beta$ is not available, estimating $\alpha$ with the optimal rate is challenging; so additional assumptions that imply the estimability of $\beta$ are usually made. In this paper, we propose an adaptive estimator of $\alpha$, and show that this estimator attains the rate  $\big(n/\log\log n\big)^{-\frac{\beta}{2\beta+1}}$ without a priori knowledge of $\beta$ and any additional assumptions. 
Moreover, we prove that this $\big(\log\log n\big)^{\frac{\beta}{2\beta+1}}$ factor is unavoidable by obtaining the companion lower bound. 
\par

\vspace{9pt}
\par
\end{quotation}\par


\fontsize{10.95}{14pt plus.8pt minus .6pt}\selectfont

\setcounter{chapter}{1}
\setcounter{equation}{0} 
\noindent {\bf 1. Introduction}

We consider the problem of estimating the tail index $\alpha$ of an $(\alpha, \beta)$ second-order Pareto distribution $F$, given $n$ i.i.d. observations $X_1,\ldots,X_n$. 
More precisely, we assume that for  some  $\alpha, \beta, C, C' >0$, 
\begin{equation}\label{model}
\big|1-F(x) -Cx^{-\alpha}\big| \leq C' x^{-\alpha(1 + \beta)}.
\end{equation}
We will write $\mathcal S(\alpha,\beta):=\mathcal S(\alpha,\beta,C,C')$ for the set of distributions that satisfy this property (see Definition~\eqref{defn:secondorder}).
Here the tail index $\alpha$ characterizes the heaviness of the tail, and $\beta$ represents the proximity between $F$ and an $\alpha$-Pareto distribution $F_\alpha^P: x \in [C^{1/\alpha}, \infty) \rightarrow 1-Cx^{-\alpha}$.

There is an abundant literature on the problem of estimating $\alpha$. A very popular estimator is Hill's estimator~\citep{hill1975simple} (see also Pickands'  estimator \citep{pickands1975statistical}). 
\cite{hill1975simple} considered $\alpha$-Pareto distribution for the tail, and suggested an estimator $\hat \alpha_H(r)$ of the tail index $\alpha$ based on the order statistics  $X_{(1)} \leq \ldots \leq X_{(n)}$  where $r$ is the fraction of order statistics from the tail,
\begin{equation}\label{hill_estimate}
\hat \alpha_H(r) = \left(\frac{1}{\lfloor r n\rfloor} \sum_{i = 1}^{\lfloor r n\rfloor}  \frac{\log(X_{(n-i+1)})}{\log(X_{(n-\lfloor r n\rfloor+1)})} \right)^{-1}.
\end{equation}
For more details, see e.g. \cite{dehaan2006}.

Limiting distribution of Hill's estimator was first proved by \cite{hall1982} when $\beta$ is known.
Under a model that is quite similar to~(\ref{model}), he proved that if $rn^{1/(2\beta+1)} \rightarrow 0$ as $n \rightarrow \infty$, $\sqrt{nr}(\hat \alpha_H(r) - \alpha)$ converges in distribution to $N(0, \alpha^2)$. 
He also considered more restricted condition, say, the exact Hall condition,   
\begin{equation}\label{hallmodel}
\big|1-F(x)-Cx^{-\alpha}\big| = C'' x^{-\alpha(1+\beta)} + o(x^{-\alpha(1+\beta)}).
\end{equation}
Under the model (\ref{hallmodel}) with the choice of the sample fraction $r^* = Cn^{-\frac{1}{2\beta+1}}$ with some constant $C$, Theorem 2 of \cite{hall1982} states that 
$n^{\beta/(2\beta+1)} (\hat \alpha_H(r^*) - \alpha)
$ \color{black}converges to a Gaussian distribution with finite  mean and variance, depending on the parameters of the true distribution.

The companion lower bound $n^{-\beta/(2\beta+1)}$ under the assumption (\ref{model}) was proved by \cite{hall1984}.
\cite{drees2001} improved this result by obtaining sharp asymptotic minimax bounds again when $\beta$ is available.
From these results, we know that the second-order parameter $\beta$ is crucial to understand the behaviour of the distribution. 
Indeed, it determines the rate of estimation of $\alpha$ as well as the optimal sample fraction.

However, $\beta$ is unknown in general. To cope with this problem, 
\cite{hall1985adaptive} proved that under condition (\ref{hallmodel}), it is possible to estimate $\beta$ in a consistent way, and thus also to estimate the sample fraction $r^*$ consistently by $\hat r$ (see Theorem 4.2 in their paper). Theorem 4.1 of \cite{hall1985adaptive} deduces from these results that the estimate $\hat \alpha_H(\hat r)$ is asymptotically as efficient as $\hat \alpha_H(r^*)$, that is, $n^{\beta/(2\beta+1)} (\hat \alpha_H(\hat r) - \alpha)$ converges to a Gaussian distribution with the same mean and variance as the one  resulting from the choice $r^*$. Their result is pointwise, but not uniform under the model~\eqref{hallmodel}, as opposed to the uniform convergence when $\beta$ is known.

This first result on adaptive estimation was extended in several ways. For instance, \cite{ivette2008tail} provided more precise ways to reduce the bias of the estimate of $\alpha$ using the estimate of~$\beta$ by supposing the third order condition. The adaptive estimates of $\alpha$ under the third order condition was considered in \cite{gomes2012}. In addition, several other methods for estimating $r^*$ have been proposed, e.g.~bootstrap (e.g. \cite{danielsson2001}) or regression (e.g. \cite{beirlant1996}). In particular, \cite{drees1998} considered a method that is related to Lepski's method (see~\cite{lepski1992problems} for more details in a functional estimation setting)  by choosing the sample fraction that balances the squared bias and the variance of the resulting estimate. 
They proved that Hill's estimate computed with this sample fraction is asymptotically as efficient as the oracle estimate if $F$ satisfies a condition that is slightly more restrictive than the condition (\ref{hallmodel}). Finally, \cite{spokoiny} consider a more general setting than (\ref{model}). However, when they apply their results to the exact Hall model (without little $o$), their estimator obtains the optimal rate up to a $\log(n)$ factor, which is clearly sub-optimal as proven in~\cite{hall1985adaptive}.


In this paper, we focus on deriving results for the setting (\ref{model}). Indeed, many common distributions (in particular some distributions with change points in the tail) belong to it, and the construction of the lower bound in~\cite{hall1984} was proved in this model. However, to the best of our knowledge, either the existing results that we mentioned previously hold in a more restrictive setting than the model~(\ref{model}), typically in a model that is close to the model~\eqref{hallmodel} (see e.g.~\cite{hall1985adaptive,beirlant1996, drees1998, danielsson2001, ivette2008tail, gomes2012}), or the convergence rates for the setting (\ref{model}) in the previous results are worse than one could expect (see e.g.~\cite{spokoiny}). It is important to note here that the set of distributions described in Equation~(\ref{model}) is significantly larger than the set of distributions that satisfy the restricted condition~(\ref{hallmodel}). As will be explained later, the adaptive estimation in our setting (i.e.~condition (\ref{model})) is  more involved since the second-order parameter $\beta$ is not always estimable (even a consistent estimator does not exist for all distributions in this model), and the adaptive procedures based on estimating $\beta$ or the oracle sample fraction $r^*$ as in the papers (\cite{hall1985adaptive, ivette2008tail, gomes2012}) might not work on all the functions satisfying~(\ref{model}).

The contributions of this paper are the following. We construct an adaptive estimator $\hat \alpha$ of $\alpha$ in the setting~(\ref{model}) and prove that $\hat \alpha$ converges to $\alpha$ with the rate $ (n/\log \log(n))^{-\beta/(2\beta+1)}$.
More precisely, for an arbitrarily small $\epsilon>0$, and some arbitrarily large range $I_1$ for $\alpha$ and $[\beta_1, \infty)$ for $\beta$, there exist large constants $D,E>0$ such that for any $n> D \log(\log(n)/\epsilon)$
\begin{equation}\label{eq:qdqcoucou}
\sup_{\alpha \in I_1, \beta>\beta_1} \sup_{F \in \mathcal S(\alpha,\beta)}\mathbb{P}_F\left( |\hat \alpha -\alpha| \geq  E \left(\frac{n}{\log(\log(n)/\epsilon)}\right)^{-\frac{\beta}{2\beta + 1}}\right) \leq \epsilon.
\end{equation}
There is an additional $\big(\log\log(n)\big)^{\frac{\beta}{2\beta+1}}$ factor in the rate with respect to the oracle rate, which comes from the fact that we adapt over $\beta$ on a set of distributions where $\beta$ is not estimable. 
Although we obtain worse rates of convergence than the oracle rate, we actually
 prove the optimality of our adaptive estimator by obtaining a matching lower bound. 
Indeed, there exists a small enough constant $E'>0$ such that for any $n$ large enough, and for any estimator~$\tilde \alpha$,
$$
\sup_{\alpha\in I_1, \beta>\beta_1} \sup_{F \in \mathcal S(\alpha,\beta)}\mathbb{P}_F\left( |\tilde \alpha -\alpha| \geq  E' \left(\frac{n}{\log(\log(n))}\right)^{-\frac{\beta}{2\beta + 1}}\right) \geq \frac{1}{4}.
$$
Both lower and upper bounds containing the $\left(\log \log (n)\right)^{\beta/(2\beta+1)}$ factor are new to the best of our knowledge (we do not provide a tight scaling factor as in the paper by \cite{novak2013}, but the setting in this paper is different and their rate does not involve this additional $(\log \log (n))^{\beta/(2\beta+1)}$ factor).
The presence of the $\log \log n$ factor is not unusual in adaptive estimation (see \cite{spok} in a signal detection setting).
This issue is also discussed in the paper~\citep{drees1998}.

The adaptive estimator $\hat \alpha$ we propose in this paper is based on a sequence of estimates $\hat \alpha(k)$ defined in (\ref{estimator}), where the parameter $k\in \mathbb N$ plays a role similar to the sample fraction in Hill's estimator (see Subsection 3.1 for more details). These estimates $\hat \alpha(k)$ are not based on order statistics, but on probabilities of tail events. We first prove that for an appropriate choice of this threshold $k$ (independent of $\alpha$ or $\beta$), $\hat \alpha(k)$ is consistent. We then prove that for an oracle choice of $k$ (as a function of $\beta$), this estimate is minimax-optimal for distributions satisfying (\ref{model}) with the rate $n^{-\frac{\beta}{2\beta+1}}$. Finally an adaptive version of this estimate, where the parameter $k$ is chosen in a data-driven way without knowing $\beta$ in advance, is proved to satisfy  Equation~\eqref{eq:qdqcoucou}. 

\par

\vspace{10pt}
\setcounter{chapter}{2}
\setcounter{equation}{0} 
\noindent {\bf 2. Definitions of distribution classes}

In this section, we introduce two sets of distributions of interest, namely the class of approximately $\alpha$-Pareto distributions, and the class of approximately $(\alpha, \beta)$ second-order Pareto distributions.
We let $\mathcal D$ be the class of distribution functions on $[0,\infty)$.
\begin{defn}\label{defn:asymppareto}
Let $\alpha>0$, $C>0$. We denote by $\mathcal{A}(\alpha,C)$ the class of approximately $\alpha$-Pareto distributions:
\begin{equation*}
\mathcal{A}(\alpha, C) = \Big\{F \in \mathcal D: \lim_{x \rightarrow \infty} (1-F(x))x^{\alpha} =C \Big\}.
\end{equation*}
\end{defn}
Distributions in $\mathcal{A}(\alpha, C)$ converge to Pareto distributions for large $x$, and these distributions have been used as a first attempt to understand heavy tail behavior (see~\cite{hill1975simple, dehaan2006}). The first-order parameter $\alpha$ characterizes the tail behavior such that distributions with smaller $\alpha$ correspond to heavier tails.

In order to provide rates of convergence (of an estimator of $\alpha$), we define the set of second-order Pareto distributions.
\begin{defn}\label{defn:secondorder}
Let $\alpha>0$, $C>0$, $\beta>0$ and $C'>0$. We denote by $\mathcal{S}(\alpha, \beta, C, C')$ the class of approximately $(\alpha, \beta)$ second-order Pareto distributions:
\begin{equation}\label{eq:secondorder}
\mathcal{S}(\alpha,\beta, C,C') = \Big\{F \in \mathcal D: \forall x \ \text{s.t.} \ F(x) \in (0,1], \big|1-F(x) -Cx^{-\alpha}\big| \leq C' x^{-\alpha(1 + \beta)}  \Big\}.
\end{equation}
\end{defn}
From the definition of \ref{defn:secondorder}, we know that not only are the distributions in $\mathcal{S}(\alpha, \beta, C, C')$ approximately $\alpha$-Pareto, but we additionally have a bound on the rate at which they approximate Pareto distributions. This rate of approximation is linked to the second-order parameter $\beta$---a large $\beta$ corresponds to a distribution that is very close to a Pareto distribution (in particular, when $\beta = \infty$, it becomes exactly Pareto), and a small $\beta$ corresponds to a distribution that is well approximated by a Pareto distribution only for a very large $x$. 
From now, if there is no confusion, we call the distributions  in $\mathcal{S}(\alpha, \beta, C,C')$ second-order Pareto distributions, and we use the notation $\mathcal{A}$ and $\mathcal{S}$ without writing parameters explicitly.

The condition in (\ref{eq:secondorder}) is related to the condition (\ref{hallmodel}), but is weaker. Indeed, the condition (\ref{hallmodel}) implies
\begin{equation*}
\lim_{x \rightarrow \infty} \frac{1 - F(x)- Cx^{-\alpha}}{x^{-\alpha(1+\beta)}} = C',
\end{equation*}
 whereas our condition imposes only an upper bound, 
\begin{equation*}
\lim\sup_{x \rightarrow \infty} \Big|\frac{1 - F(x)- Cx^{-\alpha}}{x^{-\alpha(1+\beta)}}\Big| \leq C'.
\end{equation*}
This difference is essential in the estimation problem. For instance, in the setting (\ref{hallmodel}), it is possible to estimate $\beta$ consistently  (see e.g.~\cite{hall1985adaptive}), whereas in our setting (\ref{eq:secondorder}), it is not possible to estimate $\beta$ consistently over the set $\mathcal{S}$ of distributions for $\beta \in [\beta_1, \beta_2]$ with $0<\beta_1 <\beta_2$. Adaptive estimation of $\alpha$ is thus likely to be more involved in our setting than in the more restricted model (\ref{hallmodel}). For instance, many adaptive techniques rely on estimating $\beta$ or the sample fraction as a function of $\beta$, which is not directly applicable in our setting (see e.g.~\cite{hall1985adaptive,danielsson2001,gomes2012}). 

\begin{remark} The difference between the functions satisfying the condition in Definition~\ref{defn:secondorder} and the condition~\eqref{hallmodel} is related to the difference between H\"older functions that actually attain their H\"older exponent and H\"older functions that are in a given H\"older ball but do not attain their H\"older exponent (see e.g.~\cite{gine2010confidence} for a comparison of these two sets, and the problem for estimation when the second set is considered).
\end{remark}
\par

\setcounter{chapter}{3}
\setcounter{equation}{0} 
\noindent {\bf 3. Main results}

Most estimates in the literature are based on order statistics (as Hill's estimate or Pickands'  estimate), which causes a difficulty for one to analyse them in a non-asymptotic way. 
In contrast, the estimate we will present in Section 3.1 verifies large deviation inequalities  in a simple way. 
This estimate is based on probabilities of well chosen tail events.

\vspace{10pt}
\noindent{\bf 3.1. A new estimate}\label{ss:est}

Let $X_1, \ldots, X_n$ be an i.i.d. random sample from a distribution $F \in \mathcal{A}$. We write, for any $k \in \mathbb N$,
\begin{equation*}
p_k := \mathbb P(X>e^k) = 1 - F(e^k),
\end{equation*}
and its empirical estimate
\begin{equation*}
\hat p_k := \frac{1}{n} \sum_{i=1}^n \mathbf{1}\{ X_i > e^k \}.
\end{equation*}
We define the following estimate of $\alpha$ for any $k \in \mathbb N$
\begin{equation}\label{estimator}
\hat \alpha(k) := \log(\hat p_k) - \log(\hat p_{k+1}).
\end{equation}
This estimate gives the following  large deviation inequalities, which is crucial for proving consistency and convergence rates of $\hat \alpha(k)$.

\begin{lem}[Large deviation inequality]\label{lem:largedeviation}
Let $X_1, \ldots, X_n$ be an i.i.d. sample from $F$.
\ben
\item[\textbf{A.}]
Suppose $F \in \mathcal{A}$ and let $\delta>0$. For any $k$ such that $p_{k+1} \geq \frac{16 \log(2/\delta)}{n}$, with probability larger than $1-2\delta$,
\begin{align}
\big| \hat \alpha(k) -\left(\log(p_k) - \log(p_{k+1}) \right)\big| 
&\leq 6\sqrt{\frac{\log(2/\delta)}{n p_{k+1}}}. \label{eq:lardev}
\end{align}
\item[\textbf{B.}]
Assume now that $F\in \mathcal{S}$ and let $\delta>0$. For any $k$ such that $p_{k+1} \geq \frac{16 \log(2/\delta)}{n}$ and $e^{-k\alpha\beta} \leq C/(2C')$, with probability larger than $1-2\delta$,
\begin{align}
\big| \hat \alpha(k) -\alpha \big| &\leq 6\sqrt{\frac{\log(2/\delta)}{n p_{k+1}}} + \frac{3C'}{C} e^{-k\alpha\beta} \label{eq:lardev2} \\
& \leq 6\sqrt{\frac{e^{(k+1)\alpha+1} \log(2/\delta)}{C n}} + \frac{3C'}{C}e^{-k\alpha\beta}. \label{eq:coucou}
\end{align}
\een
\end{lem}

For this new estimate $\hat \alpha(k)$, $k$ plays a similar role as the sample fraction in Hill's estimate (\ref{hill_estimate}). The bias-variance trade-off should be solved by choosing $k$ in an appropriate way as a function of $\beta$ (we will explain this more in details later). Choosing a too large $k$ leads to using a small sample fraction, and the resulting estimate has a large variance and a small bias. On the other hand, choosing a too small $k$ yields a large bias and a small variance for the estimate. The optimal $k$ equalises the bias term and the standard deviation.



\vspace{10pt}
\noindent{\bf 3.2. Rates of convergence}


We first consider the set of approximately Pareto distributions, and prove that the estimate  $\hat \alpha(k_n)$  is consistent if we choose $k_n$  such that it  diverges to $\infty$ but not too fast.
\begin{thm}[Consistency in $\mathcal{A}$]\label{th:alphaclassasspar}
Let $F \in \mathcal{A}$. Let $k_n \in \mathbb N$ be such that $k_n {\rightarrow} \infty$ and $(\log(n)/n) e^{k_n \alpha} \rightarrow 0$ as $n \rightarrow \infty$. Then
\begin{equation*}
\hat \alpha(k_n)\rightarrow \alpha \ \ a.s.
\end{equation*}
Choosing (for instance) $k_n =(\log \log(n))$ ensures almost sure convergence.
\end{thm}
The estimate $\hat \alpha(\log\log(n))$  converges to $\alpha$ almost surely  under the rather weak assumption that $F$ belongs to $\mathcal{A}$. But on such sets, no uniform rate of convergence exists, and this is the reason why the restricted set $\mathcal{S}$ is  introduced.


Let $\alpha, \beta, C,C'>0$. Consider now the set $\mathcal{S}:= \mathcal S(\alpha, \beta, C,C')$ of second-order Pareto distributions. We assume in a  first instance that, although we do not have access to $\alpha$, we know the parameter $\alpha(2\beta+1)$.
It is not very realistic assumption, but we will explain soon how we can modify the estimate so that it is minimax optimal on the class of second-order Pareto distributions.
\begin{thm}[Rate of convergence when $\alpha(2\beta+1)$ is known]\label{th:alphaclassassseco}
Let $n$ be such that (\ref{ncond3.6}) is satisfied. Let $k_n^* = \lfloor \log(n^{\frac{1}{\alpha(2\beta+1)}})+1\rfloor$. Then for any $\delta>0$,
we have
\begin{equation*}
\sup_{F \in \mathcal{S}} \mathbb P_F\left( |\hat \alpha(k_n^*) -\alpha| \geq  \left( B_1 + \frac{3C'}{C} \right)n^{-\frac{\beta}{2\beta + 1}}\right) \leq 2\delta,
\end{equation*}
where $B_1 = 6\sqrt{e^{2\alpha+1}\frac{\log(2/\delta)}{C }}$.
\end{thm}

Theorem \ref{th:alphaclassassseco} states that, uniformly on the class of second-order Pareto distributions, the estimate $\hat \alpha(k_n^*)$ converges to $\alpha$ with the minimax optimal rate $n^{-\frac{\beta}{2\beta+1}}$ (see~\cite{hall1984} for the matching lower bound).

\begin{remark}
Theorem \ref{th:alphaclassassseco} can be used to prove the convergence rate of our estimator by modifying the choice of $k_n^*$, when $\alpha(2\beta+1)$ is unknown but only $\beta$ is known.
For instance, we can plug a rough estimate  $\tilde \alpha := \hat \alpha((\log \log(n))^2)$ of $\alpha$ into $k_n^*$. 
The idea behind this choice is that with sufficiently large $n$, we have with high probability, 
\begin{equation*}
|\hat \alpha((\log\log(n))^2) - \alpha| = O\left(\frac{1}{\log n}\right).
\end{equation*}
Then $\hat k_n^1$ is defined as $\lfloor \log(n^{\frac{1}{\tilde \alpha(2\beta+1)}})+1\rfloor$.
Finally, the rate of convergence of $\hat \alpha(\hat k_n^1)$ can be shown as $n^{-\beta/(2\beta+1)}$  by  proving $\exp(\hat k_n^{1}) = O(n^{1/(\alpha(2\beta+1))})$ with high probability.
\end{remark}

 However, the previous optimal choice of $k$ ($k_n^*$ or $\hat k_n^1$)  still depends on $\beta$, which is unavailable in general.  To deal with this problem, we construct an adaptive estimate of $\alpha$ that does not depend on $\beta$ but still attains a  rate that is quite close to the minimax optimal  rate ~$n^{-\frac{\beta}{2\beta+1}}$ on the class of $\beta$ second-order Pareto distributions.
 
The adaptive estimator is obtained by considering a kind of bias and variance trade-off based on the large deviation inequality (\ref{eq:lardev}).
Suppose we know the optimal choice of $k^*$. Then this $k^*$ will optimize the squared error by making bias and standard error (of the estimate with respect to its expectation) equal. Since the bias is decreasing while the standard error is increasing as $k$ increases, for all $k'$ larger than this optimal $k^*$, the bias will be smaller than the standard error.
Based on this heuristic (originally proposed by \cite{lepski1992problems}), we pick the smallest $k$ which satisfies for all $k'$ larger than $k$, the proxy for the bias is smaller than the proxy for the standard error $O(\sqrt{1/(n \hat p_{k'+1})})$ as in (\ref{eq:lardev}). \color{black}
For the proxy for the bias, we use $|\hat \alpha(k') - \hat \alpha(k)|$ by treating $\hat \alpha(k)$ as the true $\alpha$ based on the idea that $\hat \alpha(k)$ would be very close in terms of the rate to the true $\alpha$ (if $k$ is selected in an optimal way).

More precisely, we choose $k$ as follows, for $1/4>\delta>0$
\begin{align}
\hat k_n =& \inf \Big\{k \in \mathbb N: \hspace{2mm} \hat p_{k+1} > \frac{24\log(2/\delta)}{n} \ \ \mathrm{and} \nonumber\\
 &\forall k'>k \hspace{2mm}\mathrm{s.t.}\hspace{2mm} \hat p_{k'+1} >\frac{24\log(2/\delta)}{n}, \ \  |\hat \alpha(k') - \hat \alpha(k)| \leq A(\delta) \sqrt{\frac{1}{n\hat p_{k'+1}}} \Big\}, \label{eq:constkn}
\end{align}
where $A(\delta)$ satisfies the condition (\ref{thm3.8cond2}) in the following theorem.

\begin{thm}[Rates of convergence with unknown $\beta$]\label{th:alphaclassasssecoadapt}
Let $1/4>\delta>0$ and let $n$ be such that (\ref{thm3.8cond1}) is satisfied. 
Consider the adaptive estimator $\hat \alpha(k_n)$ as described in (\ref{eq:constkn}) where $A(\delta)$ satisfies the following condition
\begin{equation}
A(\delta) \geq 6 \sqrt{2(C+C')\log (2/\delta)} \left( 2\sqrt{\frac{e^{2\alpha+1}}{C}} + \frac{C'}{C}\right). \label{thm3.8cond2}
\end{equation}
Then we have
\begin{equation*}
\sup_{F \in \mathcal S} \mathbb P_F\left( |\hat \alpha(\hat k_n) -\alpha| \geq  \Big( B_2 +\frac{3C'}{C}\Big) \left(\frac{n}{\log (2/\delta)}\right)^{-\frac{\beta}{2\beta+1}} \right)
\leq \left( 1+\frac{1}{\alpha} \log \left( \frac{(C+C')n}{16}\right)\right)\delta.
\end{equation*}
where $B_2 = \Big(B_1+ 2A(\delta)\sqrt{\frac{e^{2\alpha}}{C}}\Big)\frac{1}{\sqrt{\log(2/\delta)}}$ and $B_1$ is defined in Theorem \ref{th:alphaclassassseco}.
\end{thm}


Theorem~\ref{th:alphaclassasssecoadapt} holds for any $(\alpha,\beta)$ provided that $n$ and $A(\delta)$ are larger than some constants depending on $\alpha, \beta,C,C'$, and on the probability $\delta$. The advantage of our adaptive estimator is that  since the threshold $\hat k_n$ is chosen adaptively to the samples, the second-order parameter $\beta$ does not need to be known in the procedure in order to obtain the convergence rate of $\hat \alpha(\hat k_n)$. 
Theorem~\ref{th:alphaclassasssecoadapt} gives immediately the following corollary.
\begin{cor}\label{cor:bou}
Let $\epsilon \in (0,1)$ and $C'>0$ and let $0<\alpha_1 < \alpha_2$ and $0<C_1 < C_2$. Let $A(\delta(\epsilon))$ be chosen as in Equation~\eqref{eq:Ad2}. 
If $n$ satisfies (\ref{newsamplesize}), then
\begin{align*}
\sup_{\begin{array}{c} \scriptstyle \alpha \in [\alpha_1, \alpha_2], \beta \in [\beta_1, \infty] \\  \scriptstyle C \in [C_1, C_2] \end{array}}  \sup_{ F \in \mathcal{S}(\alpha, \beta, C,C')} \mathbb P_F\left( |\hat \alpha(\hat k_n) -\alpha| \geq   B_3 \left(\frac{n}{\log \Big(\frac{2}{\epsilon} \Big(1+\frac{\log((C_2+C')n)}{\alpha_1} \Big)\Big)}\right)^{-\frac{\beta}{2\beta+1}} \right) \leq \epsilon,
\end{align*}
where $B_3$ is a constant explicitly expressed in (\ref{constantB3}), which only depends on $\alpha_2, C_1, C_2,$ and $C'$.
\end{cor}
In other words, if we fix the range of the $\alpha$ and $C$ and a lower bound on $\beta$ to which we wish to adapt, we can tune the parameters of the adaptive choice of $\hat k_n$ so that we adapt to the maximal $\beta$ such that $F$ is $\beta$ second-order Pareto. Moreover, this adaptive procedure works uniformly well over the set of second-order Pareto distributions satisfying (\ref{model}) (for $\alpha \in [\alpha_1, \alpha_2], \beta \in [\beta_1, \infty], C \in [C_1, C_2]$), which is much larger than the class of distributions that verify the condition (\ref{hallmodel}). Then this gives \textit{non-asymptotic guarantees} with \textit{explicit bounds}. 

It seems that we lose a $(\log\log(n))^{\frac{\beta}{2\beta+1}}$ factor with respect to the optimal rate, due to adaptivity to $\beta$. However, the lower bound below implies that this $(\log\log(n))^{\frac{\beta}{2\beta+1}}$ loss is inevitable; hence the rate provided in Theorem~\ref{th:alphaclassasssecoadapt} is sharp.



\begin{thm}[Lower bound]\label{thm:lb}
Let $\alpha_1,\beta_1, C_1, C_2, C' > 0$ be such that $C_1\leq \exp(-\frac{1}{2\alpha_1(2\beta_1+1)})$, $C_2 \geq 1$ and $C' \geq \frac{1}{2\alpha_1\beta_1}$. Let $n \geq \exp(16)$ satisfy the condition (\ref{nassumption}). Then for any estimate $\tilde \alpha$ of $\alpha$,
$$
\sup_{\begin{array}{c} \scriptstyle \alpha \in [\alpha_1, 2\alpha_1], \beta \in [\beta_1, \infty) \\  \scriptstyle C \in [C_1, C_2] \end{array}}  \sup_{ F \in \mathcal{S}(\alpha, \beta, C,C')}  \pp_{F} \left(|\tilde \alpha-\alpha| \geq  B_4 \Big(\frac{n}{\log\big(\log (n)/2\big)}\Big)^{-\beta/(2\beta+1)} \right) \geq  \frac{1}{4},
$$
where $B_4$ is a constant depending on $\alpha_1$ and $\beta_1$, which is provided in (\ref{constantscond}).
\end{thm}

\vspace{10pt}
\noindent{\bf 3.3. Additional remarks on our estimate}\label{addi_remark}

In the definition of our estimate, we use exponential spacings (i.e.~we estimate the probability that the random variable is larger than $e^k$),  but we can generalize our estimate by considering the probability of other tail events.  For some parameters $u> v\geq 1$, define
\begin{equation*}
\hat q_u = \frac{1}{n} \sum_{i=1}^n \mathbf{1}\{ X_i > u \}, \hspace{2mm}\mathrm{and}\hspace{2mm} \hat q_v = \frac{1}{n} \sum_{i=1}^n \mathbf{1}\{ X_i > v \}.
\end{equation*}
We define the following estimate of $\alpha$ as
\begin{equation}\label{eq:viviviviiv}
\hat \alpha(u,v) = \frac{\log(\hat q_v) - \log(\hat q_{u})}{\log(u) - \log(v)}.
\end{equation}
If we fix $v \sim O(n^{1/(\alpha(2\beta+1)})$ and $u/v \sim O(1)$, then we will also obtain the oracle rate  for estimating $\alpha$ with $\hat \alpha(u,v)$.  
However, the choice of $u/v$ will have an impact on the constants. In practice, these parameters are important to tune well (in particular for the exact Pareto case, or for distributions satisfying Equation~\eqref{hallmodel}). However, a precise analysis of the best choices for $u$ and $v$ (in terms of constants) is beyond the scope of this paper. 

Another point we want to address is the relation between our estimate and usual estimates based on order statistics. 
To estimate the tail index $\alpha$, it is natural to consider the quantiles associated with the tail probabilities.
For the estimates based on order statistics, one fixes some tail-probabilities and then observes the order statistics in order to estimate the quantiles. On the other hand, we fix some values corresponding to the quantiles, and estimate the associated tail probabilities. Based on such a link, one could relate any existing method based on order statistics to the method based on tail probabilities. 

In particular, the estimator based on order statistics corresponding to our estimator would be of the form, for some parameters $1 \geq q_v > q_u \geq 0$,
\begin{align}\label{eq:viviviviivo}
\tilde \alpha(q_u, q_v) = \frac{\log(q_v) - \log(q_u)}{\log(\hat u) - \log(\hat v)},
\end{align}
where $\hat u = X_{(n - \lfloor q_u n \rfloor)}$ and $\hat v = X_{(n - \lfloor q_v n \rfloor)}$. 
This estimate can be interpreted as the inverse of some generalized Pickands'  estimate (see~\cite{pickands1975statistical}, it is however \textit{not} Pickands'  estimate). There is actually a duality between these two estimators: for any couple $(q_u, q_v)$ in the definition~\eqref{eq:viviviviivo}, it is possible to find $(u,v)$ in the definition~\eqref{eq:viviviviiv} such that these two estimates exactly match (see Figure~\ref{figcoco} for an illustration). However, there is no analytical transformation from one estimate to the other since such a transformation will be data dependent.

\begin{figure*}[h!]
\begin{center}
\includegraphics[width=8cm]{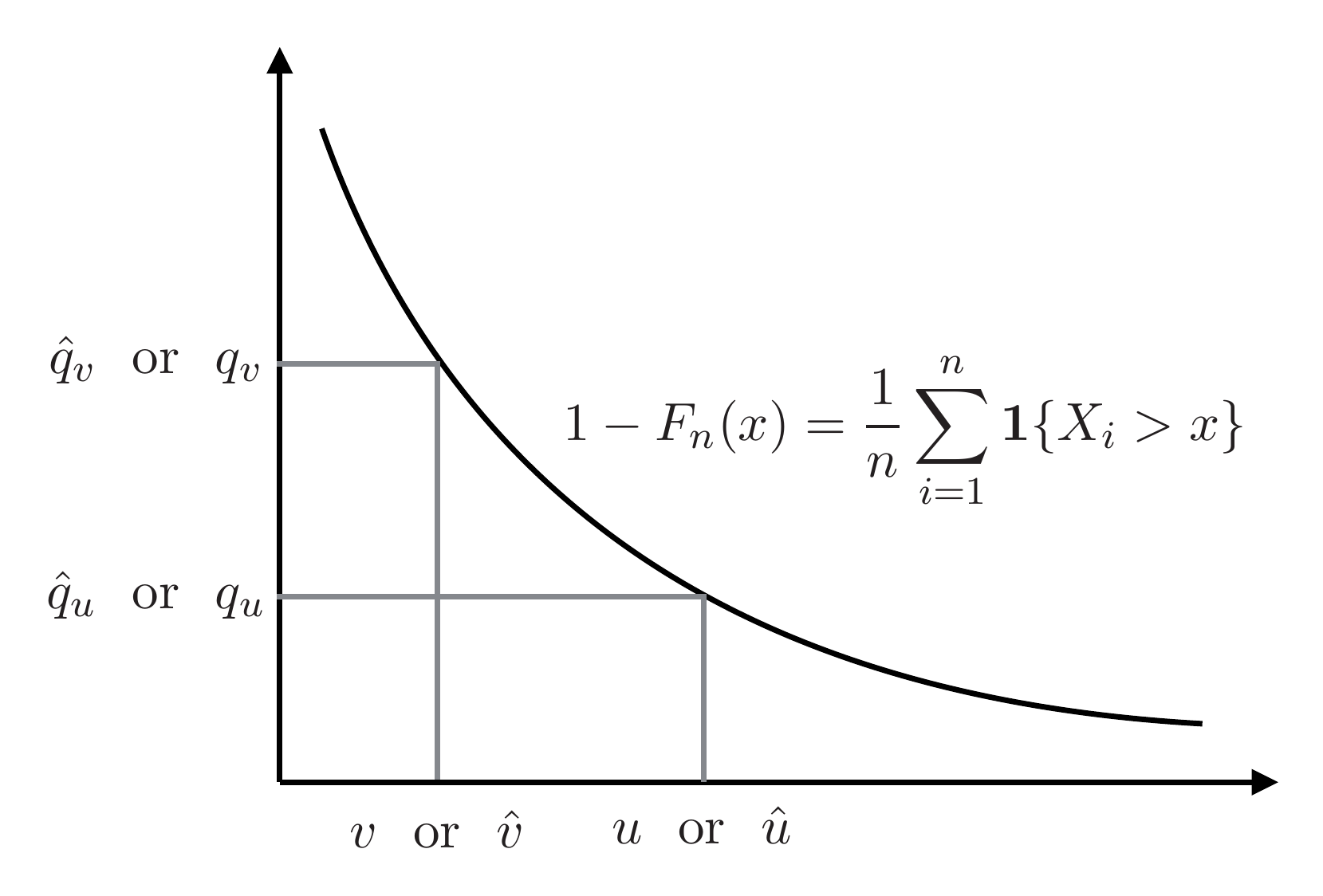}
\caption{Duality between the estimate~\eqref{eq:viviviviiv} and the estimate~~\eqref{eq:viviviviivo}.} \label{figcoco}
\end{center}
\end{figure*}

\par

\vspace{10pt}
\setcounter{chapter}{4}
\setcounter{equation}{0} 
\noindent {\bf 4. Technical proofs}

Lemma \ref{lem:bernstein} contains a classical and simple, yet important result for the paper.
\begin{lem}[Bernstein inequality for Bernoulli random variables]\label{lem:bernstein}
Let $X_1, \ldots, X_n$ be an i.i.d. observations from $F$, and we define $p_k = 1-F(e^k)$ and $\hat p_k = \frac{1}{n} \sum_{i=1}^n \mathbf{1} \{X_i >e^k\}$.
Let $\delta>0$ and also let $n$ be large enough so that $p_k \geq \frac{4\log (2/\delta)}{n}$.
Then with probability $1-\delta$,
\begin{equation}\label{eq:berndev}
 |\hat p_k - p_k | \leq 2 \sqrt{\frac{p_k \log(2/\delta)}{n}}.
\end{equation}
\end{lem}
\begin{proof}[Proof of Lemma \ref{lem:bernstein}]
The proof is using Bernstein inequality (e.g. see Lemma 19.32 of \cite{vaartbook}) of the following form; for any bounded, measurable function $g$, we have for every $t>0$,
$$
\mathbb{P} \left( \left|\sqrt{n}\left( \frac{1}{n}\sum_{i=1}^n g(X_i) - \mathbb Eg(X) \right)  \right|>t \right) \leq 2 \exp \left( -\frac{1}{4}\frac{t^2}{\mathbb Eg^2 + t||g||_\infty/\sqrt{n}} \right).
$$
We use $g(\cdot) = \mathbf{1}\{ \cdot > e^k \}$ and $t = 2\sqrt{p_k \log(2/\delta)}$ in the above inequality. Using the fact that $t =  2\sqrt{p_k \log(2/\delta)} \leq \sqrt{n}p_k$ by the assumption of $p_k \geq (4\log(2/\delta))/n$, we have
\begin{align*}
\mathbb{P}\left( \sqrt{n} |\hat p_k - p_k| > t\right) &\leq 2 \exp \left(-\frac{1}{4} \frac{t^2}{p_k + t/\sqrt{n}} \right)  \\
&\leq 2\max \left[ \exp\left(-\frac{1}{4} \frac{t^2}{p_k}\right), \exp \left( -\frac{1}{4} \sqrt{n}t \right) \right] \\
&\leq 2 \exp\left( -\frac{1}{4} \frac{t^2}{p_k} \right) \\
&= \delta,
\end{align*}
where the last equality follows by definition of $t$.
\end{proof}

\begin{proof}[Proof of Lemma \ref{lem:largedeviation}]

\textbf{A.}
Since $p_k \geq 16 \log(2/\delta)/n$, we can use Lemma \ref{lem:bernstein}. 
Rewriting the inequality~\eqref{eq:berndev}, we have with probability larger than $1-\delta$ 
\begin{equation*}
\log\left(1 - 2\sqrt{\frac{\log(2/\delta)}{np_k}}\right) \leq \log(\hat p_k) - \log (p_k) \leq  \log\left(1 + 2\sqrt{\frac{ \log(2/\delta)}{np_k}}\right).
\end{equation*}
Then using the simple inequalities $\log(1+u) \leq u$, and $\log (1-u) \geq (-3u)/2$ for $u<1/2$,
\begin{equation*}
\log(p_k) - 3\sqrt{\frac{\log(2/\delta)}{n p_k}} \leq \log(\hat p_k)  \leq  \log(p_k) + 2\sqrt{\frac{ \log(2/\delta)}{n p_k}}.
\end{equation*}
By using a similar inequality for $\log (\hat p_{k+1})$, with probability larger than $1- 2 \delta $,
\begin{align}
\big| \hat \alpha(k) -\left(\log(p_k) - \log(p_{k+1}) \right)\big| &\leq 3\sqrt{\frac{\log(2/\delta)}{n p_k}} + 3\sqrt{\frac{\log(2/\delta)}{n p_{k+1} }} \nonumber\\
&\leq 6\sqrt{\frac{\log(2/\delta)}{n p_{k+1}}}.\label{eq:deviou}
\end{align}

\textbf{B.}
By definition of second-order Pareto distributions, we have
$\big| p_{k} - Ce^{-k\alpha}\big| \leq C' e^{-k\alpha(1 + \beta)}$,
or equivalently, 
$$
\left| \frac{e^{k\alpha}p_k}{C} - 1 \right| \leq \frac{C'}{C}e^{-k \alpha \beta}.
$$
Since we assume $\frac{C'}{C} e^{-k\alpha\beta} \leq 1/2$, we have
\begin{equation*}
\big| \log(p_{k}) - \log(C) + k\alpha\big| \leq \frac{3C'}{2C} e^{-k\alpha\beta}.
\end{equation*}
A similar result also holds for $p_{k+1}$, and thus
\begin{equation}\label{eq:bia}
\big| \log(p_{k}) - \log(p_{k+1})  - \alpha\big| \leq \frac{3C'}{C} e^{-k\alpha\beta}.
\end{equation}

Combining Equations~\eqref{eq:deviou} and~\eqref{eq:bia}, we obtain the large deviation inequality (\ref{eq:lardev2}).
%
Now, using the property of the second-order Pareto distributions, we can bound $p_{k+1}$ from below.
\begin{align*}
p_{k+1} &\geq Ce^{-(k+1)\alpha} \left( 1- \frac{C'}{C} e^{-(k+1)\alpha \beta} \right)\\
&\geq \frac{C}{2} e^{-(k+1)\alpha} \geq C e^{-(k+1)\alpha-1},
\end{align*}
where the second inequality comes from the assumption that $e^{-k\alpha\beta} \leq C/(2C')$.
By substituting this into the inequality (\ref{eq:lardev2}), the final inequality (\ref{eq:coucou}) follows.
\end{proof}

\begin{proof}[Proof of Theorem~\ref{th:alphaclassasspar}]
The proof consists of the two steps---bounding the bias, and bounding the deviations of the estimate---as in the proof of the Lemma \ref{lem:largedeviation}.B. 

First, we bound the bias (more precisely, a proxy for the bias) using the property of the distribution class $\mathcal{A}$. 
By definition, we know that for any $\epsilon$ such that $C/2>\epsilon>0$, there exists a constant $B>0$ such that for $x>B$,
\begin{equation*}
\big| 1 - F(x) - Cx^{-\alpha}\big| \leq \epsilon x^{-\alpha}.
\end{equation*}
Since $k_n \rightarrow \infty$ as $n \rightarrow \infty$, for any $n$ larger than some large enough $N_1$ (i.e.~such that $\forall n\geq N_1$, $e^{k_{n}} >B$), we have
\begin{equation}\label{eq:beurki}
\big| p_{k_n} - Ce^{-k_n\alpha}\big| \leq \epsilon e^{-k_n\alpha},
\end{equation}
which yields since $\epsilon<C/2$,
$
\big| \log(p_{k_n}) - \log(C) + k_n\alpha\big| \leq \frac{3\epsilon}{2C}
$
using the same technique as for the proof of Lemma \ref{lem:largedeviation}.
This holds also for $k_{n}+1$ and thus
\begin{equation}\label{eq:coconut}
\big| \log(p_{k_n}) - \log(p_{k_n+1})  - \alpha\big| \leq \frac{3\epsilon}{C}.
\end{equation}

Note also that Equation~\eqref{eq:beurki} can be used to bound the $p_{k_n+1}$ below as follows. 
\begin{equation} \label{eq:bulb1}
p_{k_n+1} \geq (C-\epsilon)e^{-(k_n+1)\alpha} \geq \frac{C}{e^{\alpha+1}} e^{-k_n\alpha}.
\end{equation}
Since $(\log(n) e^{k_n \alpha})/n \rightarrow 0$ as $n \rightarrow \infty$, we know that there exists $N_2$ large enough, such that for any $n \geq N_2$, 
$p_{k_n +1} \geq 32 \log(n)/n$.

Then we can bound the proxy for the standard deviation using the result~\eqref{eq:lardev} in Lemma \ref{lem:largedeviation}.A. 
For $n\geq \max(N_1,N_2)$, combining Equation~\eqref{eq:coconut} and Equation~\eqref{eq:lardev} with $\delta = 2/n^2$,
we have with probability larger than $1- 4 /n^2$,
$$
\big| \hat \alpha(k_n) -\alpha \big| \leq 6\sqrt{\frac{\log(n^2)}{n p_{k_n+1} }} + \frac{3\epsilon}{C}.
$$
Then we bound the first term in the right side of the above inequality using (\ref{eq:bulb1}). 
That is,
$$
6\sqrt{\frac{\log(n^2)}{n p_{k_n+1} }} 
\leq 6\sqrt{e^{\alpha+1}\frac{\log(n^2)}{Cne^{-k_n \alpha} }}
\leq \frac{6 e^{(\alpha/2)+1}}{\sqrt{C}}\sqrt{\frac{\log(n)e^{k_n \alpha}}{n}}
$$
 
By the assumption that $(\log(n) e^{k_n \alpha})/n \rightarrow 0$, and since the above inequality holds for any $\epsilon>0$, we conclude that $\alpha_n$ converges in probability to $\alpha$. Moreover, since $ \sum_n (4/n^2) < \infty$,  Borel--Cantelli  Lemma says that $\hat \alpha(k_n)$ converges to $\alpha$ almost surely.
\end{proof}

\begin{proof}[Proof of Theorem~\ref{th:alphaclassassseco}]
Let $n$ satisfy the following,
\begin{equation}\label{ncond3.6}
n>\max\Big((\frac{2C'}{C})^{\frac{2\beta+1}{\beta}},   (\frac{32\log(2/\delta) e^{2\alpha}}{C})^{\frac{2\beta+1}{2\beta}}\Big).
\end{equation}
We let $k^* = k_n^*$ such that $k_n^* :=\left\lfloor \log n^{\frac{1}{\alpha(2\beta+1)}}+1 \right\rfloor$. Note that for $n$ larger than $(2C'/C)^{\frac{2\beta+1}{\beta}}$, we have $e^{-k^*\alpha\beta}\leq C/(2C')$. This implies, together with the second-order Pareto assumption, 
$$
p_{k^*+1} \geq \frac{C}{2}n^{-\frac{1}{2\beta+1}}e^{-2\alpha} \geq  \frac{16\log(2/\delta)}{n}
$$ where the last inequality follows by assuming $n \geq (\frac{32\log(2/\delta) e^{2\alpha}}{C})^{\frac{2\beta+1}{2\beta}}$.

By (\ref{eq:coucou}) and by the choice of $k_n$, we have with probability larger than $1- 2 \delta$,
\begin{align*}
\big| \hat \alpha(k^*) -\alpha \big| &\leq  \left(6\sqrt{e^{2\alpha+1}\frac{\log(2/\delta)}{C }} + \frac{3C'}{C}\right) n^{-\frac{\beta}{2\beta+1}}.
\end{align*}
\end{proof}

The following lemma is going to be a useful tool for the proof of Theorem~\ref{th:alphaclassasssecoadapt}.

\begin{lem}\label{lem:stoc}
We define $K$ such that $p_K \geq \frac{16\log(2/\delta)}{n}$ and also $p_{K+1} <  \frac{16\log(2/\delta)}{n}$.
Then for any $k \geq K+1$, with probability larger than $1-\delta$,
\begin{equation}\label{stocdom}
\hat p_{k} \leq \frac{24 \log (2/\delta)}{n}.
\end{equation}
\end{lem}

\begin{proof}[Proof of Lemma \ref{lem:stoc}]
We let $q := 16\log(2/\delta)/n$ and define a Bernoulli random variable $Y_i(q)$ (independent from $X_1, \ldots, X_n$) where $P(Y_i(q)=1) = q$ for $i=1, \ldots, n$. 
Then we compare $m_q := \frac{1}{n}\sum_{i=1}^n Y_i(q)$ and $\hat p_{K+1} = \frac{1}{n}\sum_{i=1}^n \mathbf{1}\{X_i > e^{K+1}\}$.
Since $q > p_{K+1}$, the distribution of $\hat p_{K+1}$ is stochastically dominated by the distribution of $m_{q}$ (that is, $P(\hat p_{K+1} >t) \leq P(m_q>t)$).
By Lemma \ref{lem:bernstein}, we have with probability larger than $1-\delta$,
$$
|m_q - q| \leq 2\sqrt{\frac{q\log(2/\delta)}{n}} = \frac{8\log(2/\delta)}{n}.
$$
Then by stochastic dominance, with probability $1-\delta$, 
$$\hat p_{K+1} \leq q+2\sqrt{\frac{q\log(2/\delta)}{n}} = \frac{24\log(2/\delta)}{n}.
$$ 
Thus, for any $k \geq K+1$ using the monotonicity of $\hat p_{k}$ (that is, $\hat p_{k}\geq \hat p_{k+1}$),  we obtain that~(\ref{stocdom}) holds with probability larger than $1-\delta$ as required.
\end{proof}

\begin{proof}[Proof of Theorem~\ref{th:alphaclassasssecoadapt}]
Let $F \in \mathcal{S}(\alpha, \beta,C,C')$ and $1/4>\delta>0$.
Also we let $n$ satisfy the following,
\begin{equation}\label{thm3.8cond1}
n > \log\big(\frac{2}{\delta}\big) \max\left[32\Big(\frac{2C'}{C^{1+\beta}}\Big)^{1/\beta}, \Big(\frac{2C'}{C}\Big)^{\frac{2\beta + 1}{\beta}}, \Big(32\frac{e^{2\alpha}}{C}\Big)^{\frac{2\beta+1}{2\beta}}, \Big(\frac{96e^{2\alpha}}{C}\Big)^{\frac{2\beta+1}{\beta}}\right].
\end{equation}

The proof is based on defining an event of high probability $\xi$ where the empirical probabilities $\hat p_k$ verify a large deviation inequality for a given subset of indices $k$. We prove that conditional on this event $\xi$, the adaptive choice $\hat k_n$ defined in Theorem~\eqref{eq:constkn} is not far from the optimal choice $k^*$. 

\vspace{5pt}
{\bf Step 1: Definition of an event of high probability.}

First, we define $K\in \mathbb N$ such that $p_K \geq \frac{16\log(2/\delta)}{n}$ and also $p_{K+1} <  \frac{16\log(2/\delta)}{n}$.
By inverting the condition for the second-order Pareto distributions, $\frac{16\log(2/\delta)}{n} \leq p_K \leq (C+C')e^{-K\alpha}$ gives
$$K \leq \frac{1}{\alpha}
 \log \left( \frac{(C+C')n}{16\log(2/\delta)}\right).$$
Set $u = \frac{1}{\alpha} \log \left( \frac{Cn}{32\log(2/\delta)} \right)-1$. Then since $n>32(\frac{2C'}{C^{1+\beta}})^{1/\beta}\log(2/\delta)$, we know by definition of $\mathcal S$ that $F(e^{u+1}) > \frac{16\log(2/\delta)}{n}$. This implies in particular, since $F(e^x)$ is a decreasing function of $x$ and since $\frac{16\log(2/\delta)}{n}>p_{K+1}$, that $u < K$. This and the above upper bound give
\begin{equation}\label{rangeK}
\frac{1}{\alpha} \log \left( \frac{Cn}{32\log(2/\delta)} \right)-1 < K \leq \frac{1}{\alpha}
 \log \left( \frac{(C+C')n}{16\log(2/\delta)}\right).
\end{equation}

Second, we define $\bar k_n = \bar k \in \mathbb{N}$ such that
\begin{align*}
 \bar k := \left\lfloor \log \left( \big(\frac{n}{\log(2/\delta)}\big)^{\frac{1}{\alpha(2\beta+1)}} \right)+1 \right\rfloor.
\end{align*}

By definition of $\bar k$, we know that $\bar k <  K$. Indeed, by (\ref{rangeK}),
$$
\bar k \leq \log \left( \Big(\frac{n}{\log(2/\delta)}\Big)^{\frac{1}{\alpha(2\beta+1)}} \right) +1 \leq \frac{1}{\alpha} \log \left( \frac{Cn}{32\log(2/\delta)}\right) -1 < K
$$
where the second inequality follows by the assumption $n \geq (32\frac{e^{2\alpha}}{C})^{\frac{2\beta+1}{2\beta}} \log(2/\delta)$.
Thus, 
\begin{equation}\label{expKab}
e^{-K\alpha \beta} \leq e^{-\bar k \alpha \beta} \leq C/(2C'),
\end{equation}\color{black}
where the second inequality follows since $n>\log(2/\delta) (\frac{2C'}{C})^{\frac{2\beta + 1}{\beta}}$.

Note also that $\bar k \leq k^*$, where $k^* := \left\lfloor \log\left(n^{\frac{1}{\alpha(2\beta+1)}}\right) +1 \right\rfloor$ as before.

We define the following event
\begin{equation}\label{dfn:xi}
\xi = \Big\{\omega:  \forall k \leq K, \big| \hat p_k(\omega) - p_k \big|   \leq  2\sqrt{\frac{p_k\log(2/\delta) }{n}} ,  \hat p_{K+1}(\omega) \leq \frac{24\log(2/\delta)}{n} \Big\}.
\end{equation}
By definition, we have $p_K \geq \frac{16\log(2/\delta)}{n}$, which gives the Bernstein inequality (\ref{eq:berndev}) with probability $1-\delta$ for $k \leq K$. In addition, Lemma \ref{lem:stoc} gives (\ref{stocdom}) with probabiltiy $1-\delta$. 
Thus, an union bound implies that $\mathbb P(\xi) \geq 1 - (K+1)\delta$. 
By monotonicity of $\hat p_k$, we have on the event $\xi$, for any $k \geq K+1$,
$\hat p_{k}  \leq \frac{24\log(2/\delta)}{n}$.
This implies that on the event $\xi$, the $k,k'$ considered in Equation~\eqref{eq:constkn} are smaller than $K$ and in particular, we have $\hat k_n \leq K$.


\vspace{5pt}
{\bf Step 2: Large deviation inequality for the index $\bar k$ on $\xi$.}


If $k < K$ satisfies $e^{-k\alpha\beta} \leq C/(2C')$, then since $p_{k+1} \geq p_{K} \geq (16\log(2/\delta))/n$, \color{black}then using the exactly same proof as for Lemma \ref{lem:largedeviation}.B\color{black}, we have on $\xi$ that 
\begin{equation}\label{okay}
| \hat \alpha(k)-\alpha| \leq 6\sqrt{\frac{e^{(k+1)\alpha+1} \log(2/\delta)}{C n}} + \frac{3C'}{C}e^{-k\alpha\beta}.
\end{equation}



Since $e^{-\bar k \alpha \beta} \leq C/(2C')$ by (\ref{expKab}) and $\bar k < K$, Equation~\eqref{okay} is verified for $\bar k$ on $\xi$. Then by definition of $\bar k$ in Equation \eqref{okay}, we have on $\xi$ that
\begin{align}
|\hat \alpha(\bar k) - \alpha| \leq \left(6\sqrt{\frac{e^{2\alpha+1}}{C}}+\frac{3C'}{C} \right) \left(\frac{n}{\log(2/\delta)} \right)^{-\frac{\beta}{2\beta+1}}. \label{largedevbark}
\end{align}

Also, we have on $\xi$, using $\bar k \leq K-1$ and $p_{\bar k+1} \geq p_K \geq (16\log(2/\delta))/n$,
\begin{equation*}
\hat p_{\bar k+1} \geq p_{\bar k+1}\left(1 - 2\sqrt{\frac{\log(2/\delta)}{np_{\bar k+1}}} \right) \geq \frac{p_{\bar k+1}}{2} \end{equation*}
Then using the second order Pareto property with $(C'/C)e^{-\bar k \alpha \beta} \leq 1/2$, we have $p_{\bar k+1} \geq (Ce^{-(\bar k+1)\alpha})/2$, which gives
\begin{equation}\label{pbark}
\hat p_{\bar k+1} \geq \frac{Ce^{-(\bar k+1)\alpha}}{4} \geq \frac{Ce^{-2\alpha}}{4} \left(\frac{\log(2/\delta)}{n}\right)^{1/(2\beta+1)},
\end{equation}
where the second inequality follows from $n>\log(2/\delta) (\frac{2C'}{C})^{\frac{2\beta + 1}{\beta}}$ and from the definition of $\bar k$. 
Since $n > \left(\frac{96e^{2\alpha}}{C}\right)^{\frac{2\beta+1}{\beta}}\log(2/\delta)$, we have shown that $\hat p_{\bar k+1}$ is larger than $24\frac{\log(2/\delta)}{n}$ on $\xi$, and $\bar k$ is a candidate in the construction of $\hat k_n$. In other words, on $\xi$, $\hat p_{\bar k} \geq \frac{24\log(2/\delta)}{n}$, and by definition of $\hat k_n$ in Equation~\eqref{eq:constkn}, $\hat \alpha (\bar k)$ will be compared to $\hat \alpha (\hat k_n)$ in the construction of $\hat k_n$.


\vspace{5pt}
{\bf Step 3: Proof that $\hat k_n \leq \bar k$ on $\xi$}

Suppose that $\hat k_n > \bar k$. By definition of $\hat k_n$, on $\xi$, there exists $k > \bar k$ such that $\hat p_{k+1} > \frac{24\log(2/\delta)}{n}$ (this imposes $k < K$ on $\xi$) and
\begin{equation}\label{eq:setp1}
| \hat \alpha(k)- \hat \alpha( \bar k)  | >  A(\delta) \sqrt{\frac{1}{n\hat p_{k+1}}} \geq   \frac{ A(\delta)}{\sqrt{2(C+C')}} \sqrt{\frac{e^{k\alpha}}{n}},
\end{equation}
where the second inequality in the above is by definition of $\xi$,
\begin{align*}
\hat p_{k+1} \leq p_{k+1}\left(1+2\sqrt{\frac{\log(2/\delta)}{np_{k+1}}}\right) \leq \frac{3}{2}p_{k+1}
&\leq 2(C+C')e^{-k\alpha},
\end{align*}
where the penultimate inequality is obtained by $p_k \geq p_K \geq 16\log(2/\delta)/n$ (since $k \leq K$), and the last inequality follows by definition of the second order Pareto condition.

Since $k \geq \bar k+1$, we bound $e^{-k\alpha\beta} \leq e^{-\bar k\alpha\beta}  \leq C/(2C')$ by (\ref{expKab}). \color{black}
Also we have $p_{k+1} \geq \frac{16\log(2/\delta)}{n}$, since $p_{k+1} \geq p_K$.
Equation~\eqref{okay} is thus verified on $\xi$ for such $k>\bar k$. 
Now using $\sqrt{\frac{e^{k \alpha} \log(2/\delta) }{ n}} > e^{-k\alpha\beta}$ (since $k > \bar k$), we have
\begin{equation}\label{eq:setp2}
| \hat \alpha(k) - \alpha| 
\leq   \Big(6\sqrt{\frac{e^{\alpha+1} }{C }} + \frac{3C'}{C}\Big) \sqrt{\frac{e^{k \alpha} \log(2/\delta) }{ n}}.
\end{equation}

Equations~\eqref{eq:setp1} and~\eqref{eq:setp2} imply that on $\xi$,
\begin{align*}
| \hat \alpha( \bar k) - \alpha | &> \Big(\frac{A(\delta)}{\sqrt{2(C+C')}} -  \sqrt{\log(2/\delta)}\big(6\sqrt{\frac{e^{\alpha+1}}{C }} + \frac{3C'}{C}\big) \Big)\sqrt{\frac{e^{k\alpha}}{n}} \\
 &\geq  \big(6\sqrt{\frac{e^{2\alpha+1} }{C }} + \frac{3C'}{C}\big)\left(\frac{n}{\log(2/\delta)}\right)^{-\frac{\beta}{2\beta+1}},
\end{align*}
since we assume that $\frac{A(\delta)}{\sqrt{2(C+C')}} \geq  2 \sqrt{\log (2/\delta)}\big(6\sqrt{\frac{e^{2\alpha+1}}{C }} + \frac{3C'}{C} \big)$. This contradicts Equation~\eqref{largedevbark}, and this means that on $\xi$, $\hat k_n \leq \bar k$.

{\bf Step 4: Large deviation inequality for an adaptive estimator}  

By definition of $\hat k_n$, since on $\xi$, $\bar k$ is a candidate in the construction of $\hat k_n$ (since on $\xi$, $\hat p_{\bar k} \geq \frac{24\log(2/\delta)}{n}$, and by definition of $\hat k_n$ in Equation~\eqref{eq:constkn}), and since $\hat k_n \leq \bar k$ from Step 3, we have on $\xi$
\begin{align}
|\hat \alpha(\bar k) - \hat \alpha(\hat k_n)| &\leq A(\delta) \sqrt{\frac{1}{n\hat p_{\bar k+1}}}  \nonumber\\    
&\leq 2A(\delta) \sqrt{\frac{e^{2\alpha}}{C}} \left(\log\left(\frac{2}{\delta}\right) \right)^{-\frac{1}{2(2\beta+1)}}n^{-\frac{\beta}{2\beta+1}} \nonumber\\
&= 2\frac{A(\delta)}{\sqrt{\log(2/\delta)}} \sqrt{\frac{e^{2\alpha}}{C}} \left(\frac{n}{\log(2/\delta)}\right)^{-\frac{\beta}{2\beta+1}}, \label{eq:bibicoco}
\end{align}
where the second inequality follows on $\xi$ by Equation~\eqref{pbark}.

Hence, Equations~\eqref{eq:bibicoco} and~\eqref{largedevbark} imply that on $\xi$
\begin{equation*}
| \hat \alpha(\hat k_n) - \alpha | \leq   \left(\big(6\sqrt{\frac{e^{2\alpha+1} }{C }} + \frac{3C'}{C}\big)+ 2A(\delta) \sqrt{\frac{e^{2\alpha}}{C\log(2/\delta)}}\right)\left(\frac{n}{\log(2/\delta)}\right)^{-\frac{\beta}{2\beta+1}}.
\end{equation*}

Denote $B_1 = 6\sqrt{\frac{e^{2\alpha+1}}{C}\log(2/\delta)}$ and $B_2 = (B_1+2A(\delta)\sqrt{\frac{e^{2\alpha}}{C}})\frac{1}{\sqrt{\log(2/\delta)}}$.
Then since $\pp(\xi) \geq 1-(K+1) \delta$, we have shown that
\begin{align*}
\sup_{F \in \mathcal S} \mathbb P_F&\left( |\hat \alpha(\hat k_n) -\alpha| \geq  \Big( B_2 +\frac{3C'}{C}\Big) \left(\frac{n}{\log (2/\delta)}\right)^{-\frac{\beta}{2\beta+1}} \right) \\
&\leq (K+1)\delta \leq \left( \frac{1}{\alpha} \log \left( \frac{(C+C')n}{16}\right)+1 \right)\delta
\end{align*}
where the last inequality follows by (\ref{rangeK}). This concludes the proof.
\end{proof}

\begin{proof}[Proof of Corollary~\ref{cor:bou}]
Set 
\begin{align}
\epsilon &= \left(1+ \frac{1}{\alpha_1} \log \left((C_2+C')n\right) \right)\delta, \nonumber\\
A(\epsilon) &=6  \sqrt{2(C_2+C')}\left(\sqrt{\log \left( \frac{2}{\epsilon} \Big(1+\frac{\log ((C_2+C')n)}{\alpha_1}  \Big)  \right)} \Big(2\sqrt{\frac{e^{2\alpha_2+1} }{C_1}} + \frac{C'}{C_1}\Big)\right),\label{eq:Ad2}
\end{align}
and plug $\delta$ and $A(\epsilon):=A(\epsilon(\delta))$ in the adaptive method described in Theorem~\ref{th:alphaclassasssecoadapt}. Set
\begin{equation}\label{constantB3}
B_3 := 6\sqrt{\frac{e^{2\alpha_2+1} }{C_1}} + \frac{3C'}{C_1} + 24\frac{e^{2\alpha_2}}{C_1} \sqrt{2e(C_2+C')}   + 12e^{\alpha_2} \frac{C'}{C_1} \sqrt{2\frac{(C_2+C')}{C_1}}.
\end{equation}
It holds for any $\alpha \in [\alpha_1,\alpha_2]$, $C \in [C_1, C_2]$ and $\beta> \beta_1$ that the constant in Theorem~\ref{th:alphaclassasssecoadapt} can be bounded as
\begin{align*}
B_2 + \frac{3C'}{C} &=  6\sqrt{\frac{e^{2\alpha+1}}{C}} + 12 \sqrt{2\frac{e^{2\alpha}}{C}(C_2+C')}\Big(2 \sqrt{\frac{e^{2\alpha_2+1}}{C_1}} + \frac{C'}{C_1}\Big) + \frac{3C'}{C}\\
&\leq B_3,
\end{align*}
so $B_3$ is a uniform bound on the constant in Theorem~\ref{th:alphaclassasssecoadapt} for all considered values of $\alpha, C, \beta$.
Also, the uniform condition for the sample size is derived from Equation~\eqref{ncond3.6} and is
\begin{equation}\label{newsamplesize}
n > \log  \left( \frac{2}{\epsilon} \Big(1+\frac{\log((C_2+C')n)}{\alpha_1}\Big) \right) \max\left[32\Big(\frac{2\bar C'}{\bar C_1^{1+\beta_1}}\Big)^{\frac{1}{\beta_1}}, \Big(\frac{2\bar C'}{\bar C_1}\Big)^{2+\frac{1}{\beta_1}}, \Big(32\frac{e^{2\alpha_2}}{\bar C_1}\Big)^{1+\frac{1}{2\beta_1}} , \Big(\frac{96e^{2\alpha_2}}{\bar C_1}\Big)^{2+\frac{1}{\beta_1}}\right],
\end{equation}
where $\bar C_1 = \min(1, C_1)$ and $\bar C' = \max(1,C')$.
\end{proof}

\begin{proof}[Proof of Theorem~\ref{thm:lb}]\label{sec:lb}
We prove the lower bound by Fano's method.
Let $n$ be sufficiently large enough such that 
\begin{equation}\label{nassumption}
\Big(\frac{\min(\alpha_1,1/\alpha_1)}{2}\Big)^{\frac{2\beta_1+1}{\beta_1}} n > \log\Big(\frac{n}{\min\big(1, \frac{\alpha_1^2}{8}\exp(-\frac{2}{\alpha_1(2\beta_1 + 1)^2})\big)}\Big).
\end{equation}
{\bf Step 1: Construction of a finite set of  distributions}

Let $n \geq 2$. Let $\alpha>0$ and $\beta >1$. Let $\upsilon = \min\big(1, \frac{\alpha^2}{8\exp(\frac{1}{\alpha(2\beta - 1)})}\big)$. %
Let $M>1$ be an integer such that\color{black}
\begin{equation*}
\lfloor \log(n/\log(M))\rfloor +1 = M, 
\end{equation*}\color{black}
which implies since $n\geq 2$ that $\log(n)/2 < M < 2\log(n)$. Set for any integer $1 \leq i \leq M$
\begin{align*}
\beta_i &= \beta - \frac{i}{M} \\
\gamma_i &=  \frac{\beta_i}{2\beta_i+1} \Big(1 + \frac{\log(\upsilon)}{\log\log M}\Big)\\
K_i &= n^{\frac{1}{\alpha(2\beta_i + 1)}} \big(\log M\big)^{-\frac{\gamma_i}{\alpha \beta_i}} = \Big(\frac{n}{\upsilon\log(M)}\Big)^{\frac{1}{\alpha(2\beta_i + 1)}}\\
t_i &= K_i^{-\alpha \beta_i} = n^{-\frac{\beta_i}{2\beta_i + 1}} \big(\log M\big)^{\gamma_i} = \Big(\frac{n}{\upsilon\log (M) }\Big)^{-\frac{\beta_i}{2\beta_i + 1}}\\
\alpha_i &= \alpha-t_i = \alpha - n^{-\beta_i/(2\beta_i+1)}(\log (M))^{\gamma_i}.
\end{align*}

 Assume that $n$ is large enough so that 
\begin{align}
\frac{8\exp\left(\frac{2}{\alpha(2\beta-1)^2}\right)}{\alpha^2} &\leq \log (\log(n)/2) \label{first_n} \\
\frac{\min(\alpha,1/\alpha)}{2} n^{\frac{\beta_i }{2\beta_i + 1}} &> \big(\log(n/\upsilon)\big)^{\beta_i/(2\beta_i+1)+1}. \label{second_n}
\end{align}
Note that (\ref{first_n}) implies $\gamma_i>0$ for all $\gamma_i$, and (\ref{second_n}) implies $\frac{\min(\alpha,1/\alpha)}{2} n^{\frac{\beta_i }{2\beta_i + 1}} > M^{\beta_i/(2\beta_i+1)+1}$ by definition of $M(<2\log(n)<2\log(n/\upsilon))$.
Also we have $\beta_i \geq \beta-1$, $K_i>1$ by $n > \upsilon \log(M)$, and $\alpha - t_i \geq \alpha/2 =:\alpha_1$ since we choose $n$ large enough so that $\alpha > 2 n^{-\frac{\beta_i}{2\beta_i + 1}} \big(\log (M)\big)^{\beta_i/(2\beta_i+1)}$.


 Using these notation, we  introduce the distribution 
\begin{equation*}
1- F_0(x) = x^{-\alpha},
\end{equation*}
and for any integer $1 \leq i \leq M$,  we introduce \textit{perturbed versions} of the distribution $F_0$ 
\begin{equation*}
1- F_i(x) = x^{-\alpha} \mathbf{1} \{ 1\leq x \leq K_i\} + K_i^{-t_i}x^{-\alpha + t_i} \mathbf{1} \{ x > K_i\}.
\end{equation*}

{\bf Step 2: Properties of the constructed distributions}

We now provide a Lemma highlighting important characteristics of the distributions $F_i$ and their parameters.

\begin{lem}\label{lem:Fcoc}
Let $1 \leq i \leq M$ and $1 \leq j \leq M$. It holds that
\begin{equation}\label{eq:pareti2}
F_i \in \mathcal{S}\left(\alpha-t_i,\beta_i, K_i^{-t_i},\frac{1}{\alpha(\beta-1)}\right).
\end{equation}
Moreover
\begin{equation}\label{eq:boundK}
\exp\big(-\frac{1}{\alpha(2\beta-1)}\big) \leq K_i^{-t_j} \leq 1,
\end{equation}
and if $i\neq j$,
\begin{equation}\label{separation}
| \alpha_i-\alpha_j| \geq c(\beta) \max(t_i,t_j),
\end{equation}
where $c(\beta) := 1-\exp\left( \frac{-1}{2(2\beta+1)^2}\right)$.
\end{lem}



{\bf Step 3: Computation of the Kullback-Leibler (KL) divergence}

First, we compute the KL divergence between $F_0$ and $F_i$ and prove that it has the same order of the KL divergence between $F_i$ and $F_0$. Second, we prove that the KL divergence between $F_i$ and $F_j$ is of the same order or smaller than  $\max\left\{ KL(F_0,F_i), KL(F_j,F_0)\right\}$.  

We write $\{f_0, f_1, \ldots, f_{M}\}$ for the densities associated with distributions $\{F_0, F_1, \ldots, F_{M}\}$.

This first lemma in on the KL divergence between $F_i$ and $F_0$.
\begin{lem}\label{lem:F0}
Let $1 \leq i \leq M$. It holds that
\begin{align*}
\max\big(KL(F_0, F_i), KL(F_i, F_0)\big) = \frac{ 2 t_i^2 K_i^{-\alpha} }{\alpha^2}.
\end{align*}
\end{lem}

This second lemma uses the first lemma to obtain bounds on the KL divergence between $F_i$ and $F_j$.
\begin{lem}\label{lem:Fi}
Let $(i,j)\in \{1, \ldots, M\}^2$, $i \neq j$. It holds that
\begin{equation}\label{klbound}
KL(F_i,F_{j}) \leq \frac{2 \exp(\frac{1}{\alpha(2\beta - 1)})}{\alpha^2}\left( t_i^2K_i^{-\alpha} + 
t_{j}^2K_{j}^{-\alpha}\right).
\end{equation}
\end{lem}

{\bf Step 4: Combining the above results}

 Here we follow ideas in Fano's method using the above results. 
Let $\hat \alpha = \hat \alpha(X_1, \ldots, X_n)  =: \hat \alpha(X)$  be an estimator of $\alpha$.   Then we define the following discrete   random variable 
$$
Z =Z(X) := \arg\min_{j \in \{1,\ldots,M\}} |\hat \alpha(X) - \alpha_j|,
$$
which implies that $|\hat \alpha-\alpha_j| > c(\beta) t_j/2$ if  $Z \neq j$, since $|\alpha_i-\alpha_j| \geq c(\beta) \max(t_i,t_j)$ by Equation~\eqref{separation}.
Also we consider another random variable $Y$, uniformly distributed on $\{1,\ldots,M\}$ where $X| Y=j \sim F_j^n$. 
By bounding the maximum by the average,
\begin{align*}
 \max_{j \in \{1,\ldots,M\}} \mathbb{P}_{F_j} \left( |\hat \alpha-\alpha_j| \geq ct_j/2\right) &\geq \frac{1}{M}\sum_{j=1}^M \mathbb{P} \left( Z \neq j | Y=j\right)\\
&= \mathbb{P}(Z \neq Y)\\
&\geq 1-\frac{1}{\log M} \left(\frac{1}{M^2} \sum_{j,j'} KL(F^n_{j}, F^n_{{j'}}) + \log 2 \right),
\end{align*}
where the last inequality is obtained by Fano's inequality (see Section 2.1 in \cite{coverbook}, or see the Appendix for a proof of how this inequality is derived).

Using the fact that $KL(F_1^n, F_2^n) = nKL(F_1,F_2)$, and by Equation~\eqref{klbound},
\begin{align*}
\frac{1}{M^2}\sum_{j,j'} KL(F^n_{j},F^n_{j'}) 
&= \frac{n}{M^2}\sum_{j,j'} KL(F_{j},F_{j'})  \\
&\leq \frac{n}{M^2} \frac{2\exp(\frac{1}{\alpha(2\beta - 1)})}{\alpha^2}\sum_{j,j'} \left( t_j^2 K_j^{-\alpha}+ t_{j'}^2 K_{j'}^{-\alpha} \right)=\frac{n}{M}\frac{4\exp(\frac{1}{\alpha(2\beta - 1)})}{\alpha^2} \sum_j t_j^2 K_j^{-\alpha} \\ 
&= \frac{1}{M}\frac{4\exp(\frac{1}{\alpha(2\beta - 1)})}{\alpha^2}  \sum_j (\log M)^{\gamma_j (2\beta_j+1)/\beta_j} \\
&= \frac{4\exp(\frac{1}{\alpha(2\beta - 1)})}{\alpha^2}(\log M)^{1 + \frac{\log(\upsilon)}{\log\log(M)}}\\
&= \frac{4\exp(\frac{1}{\alpha(2\beta - 1)})}{\alpha^2}(\log M) \times \upsilon \leq \frac{1}{2}\log M.
\end{align*}
 where the third equality is by definition of $\gamma_j=  \frac{\beta_j}{2\beta_j+1} \Big(1 + \frac{\log(\upsilon)}{\log\log (M)}\Big)$ and the last inequality is by assuming $\upsilon \leq \frac{\alpha^2}{8\exp(\frac{1}{\alpha(2\beta - 1)})}$. 
Hence, since $n> \exp(16)$ we have
\begin{align*}
 \max_{j \in \{1,\ldots,M\}} \mathbb{P}_{F_j} \left( |\hat \alpha-\alpha_j| \geq \frac{c(\beta) t_j}{2}\right) &\geq \frac{1}{4}.
\end{align*}
 More specifically, using $c(\beta) = 1 - \exp(-\frac{1}{2(2\beta+1)^2}) \geq \frac{1}{2(2\beta+1)^2}$ and since $t_j = \big(\frac{\upsilon\log(M)}{n}\big)^{\frac{\beta_j}{2\beta_j+1}} \geq  \nu^{\frac{\beta_j}{2\beta_j+1}}\Big(\frac{\log\big((\log (n))/2\big)}{ n}\Big)^{\frac{\beta_j}{2\beta_j+1}}$, we have
\begin{align*}
 \max_{j \in \{1,\ldots,M\}} \mathbb{P}_{F_j} \left( |\hat \alpha-\alpha_j| \geq B(\alpha, \beta, \beta_j) \Big(\frac{\log\big((\log (n))/2\big)}{ n}\Big)^{\frac{\beta_j}{2\beta_j+1}}  \right) &\geq \frac{1}{4},
\end{align*}
where
\begin{equation}\label{defB}
B(\alpha, \beta, \beta_j) := \frac{1}{4(2\beta+1)^2} \min\Bigg[1,\Big(\frac{\alpha^2}{8\exp(\frac{1}{\alpha(2\beta - 1)})}\Big)^{\frac{\beta_j}{2\beta_j+1}}\Bigg].
\end{equation}\color{black}

By definition of $\{F_1, \ldots, F_{M}\}$, we have
\begin{equation*}
\{F_1, \ldots, F_{M}\} \subset \Big\{ F \in \mathcal{S}(\alpha^*, \beta^*, C, \tilde C'): \alpha^* \in [\alpha/2, \alpha], \beta^* \in [\beta-1,\beta], C \in [\tilde C_1,\tilde C_2]\Big\},
\end{equation*}
where $\tilde C_1(\alpha, \beta) := \exp\left(-\frac{1}{\alpha(2\beta-1)}\right)$, $\tilde C_2 :=1$, and $\tilde C'(\alpha, \beta) = \frac{1}{\alpha(\beta-1)}$.

Then by bounding the supremum by the maximum over the finite subset, we finally provide the following lower bound result.

\begin{align*}
\sup_{ \begin{array}{c} \scriptstyle \alpha^* \in [\alpha/2, \alpha], \beta^* \in [\beta-1,\beta] \\  \scriptstyle C \in [\tilde C_1,\tilde C_2 ] \end{array}} 
&\sup_{ 
 \scriptstyle F \in \scriptstyle \mathcal{S}(\alpha^*, \beta^*, C,\tilde C')} 
\pp_{F} \left(|\hat \alpha-\alpha^*| \geq  B(\alpha, \beta, \beta^*) \Big(\frac{\log\big((\log (n))/2\big)}{ n}\Big)^{\frac{\beta^*}{2\beta^*+1}} \right) \\
&\geq \max_{j \in \{1,\ldots,M\}} \mathbb{P}_{F_j} \left( |\hat \alpha-\alpha_j| \geq B(\alpha, \beta, \beta_j) \Big(\frac{\log\big((\log (n))/2\big)}{ n}\Big)^{\frac{\beta_j}{2\beta_j+1}}  \right)\\
 &\geq \frac{1}{4}.
\end{align*}
By changing parametrization and setting $\alpha_1 = \alpha/2$ and $\beta_1 = \beta-1$, we proved that
$$
\sup_{\begin{array}{c} \scriptstyle \alpha^* \in [\alpha_1, 2\alpha_1], \beta^* \in [\beta_1, \infty) \\  \scriptstyle C \in [C_1, C_2] \end{array}}  \sup_{ F \in \mathcal{S}(\alpha, \beta, C,C')}  \pp_{F} \left(|\hat \alpha-\alpha^*| \geq  B_4 \Big(\frac{n}{\log\big(\log (n)/2\big)}\Big)^{-\beta^*/(2\beta^*+1)} \right) \geq  1/4,
$$
where $C' = \tilde C'(2\alpha_1, \beta_1+1)$ and
\begin{equation}\label{constantscond}
C_1 = \tilde C_1(2\alpha_1, \beta_1+1), \ C_2 = 1, \ B_4 = B(2\alpha_1, \beta_1+1, \infty).
\end{equation}
This concludes the proof.
\end{proof}
\par

\begin{proof}[Proof of Lemma~\ref{lem:Fcoc}]

\textbf{Proof of Equation~\eqref{eq:pareti2}:} For $1 \leq i \leq M$, $F_i \in \mathcal{A}(\alpha - t_i,K_i^{-t_i})$ by definition. For $x> K_i$,  $F_i$ satisfies the second-order Pareto condition.  
For any $1 \leq  x \leq K_i$
\begin{align*}
\Big| 1 - F_i(x) - K_i^{-t_i}x^{-\alpha + t_i} \Big| &= \Big| x^{-\alpha} - K_i^{-t_i}x^{-\alpha + t_i} \Big| \\
&= x^{-\alpha} \Big| 1 - K_i^{-t_i}x^{t_i} \Big|\\
&\leq 2 x^{-\alpha} \Big| t_i \log(K_i/x) \Big|.
\end{align*}
 The last inequality is obtained  since $\forall u \in [0,1]$, $|e^{-u} - 1| \leq 2u$ and
\begin{equation*}
t_i \log(K_i) \leq n^{-\frac{\beta_i}{2\beta_i + 1}} \big(\log M\big)^{\gamma_i} \big(\frac{1}{\alpha(2\beta_i + 1)} \big)\log(n) \leq  \frac{1}{\alpha}n^{-\frac{\beta_i}{2\beta_i +1}}  \log(n)^{\gamma_i +1}  \leq 1
\end{equation*}
 by assuming large $n$ satisfying (\ref{second_n}). Then   for any $1 \leq  x \leq K_i$
\begin{align*}
\Big| 1 - F_i(x) - K_i^{-t_i}x^{-\alpha + t_i} \Big| &\leq 2 x^{-\alpha} K_i^{-\alpha\beta_i} \log(K_i/x)\\
&=  2x^{-\alpha} x^{-\alpha\beta_i} \Big(\frac{K_i}{x}\Big)^{-\alpha\beta_i}\log(K_i/x)\\
&\leq 2 x^{-\alpha} x^{-\alpha\beta_i} \Big(\frac{K_i}{x}\Big)^{-\alpha(\beta - 1)}\log(K_i/x)\\
&\leq \frac{1}{\alpha(\beta-1)} x^{-\alpha(\beta_i+1)},
\end{align*}
where the ultimate inequality follows from the fact that for any $u\geq1, t>0$, we have $u^{-t}\log(u)  \leq 1/(et)$. 
Thus, we have shown the first result~\eqref{eq:pareti2}.

\textbf{Proof of Equation~\eqref{eq:boundK}:} Let $1\leq j\leq M$. Since $K_1>1$ and $t_i >0$, we have
\begin{align*}
K_i^{-t_i} &\leq 1,
\end{align*}
and by definition of $M$,
\begin{align*}
K_i^{-t_i} &\geq \Big(n^{\frac{1}{\alpha(2\beta_i + 1)}} \Big)^{-n^{-\frac{\beta_i}{2\beta_i + 1}} \big(\log M\big)^{\gamma_i}}\\
&= \exp\Big(-\frac{\log(n)}{\alpha(2\beta_i+1)} n^{-\frac{\beta_i}{2\beta_i + 1}} \big(\log M\big)^{\gamma_i}\Big)\\
&\geq \exp\Big(-\frac{\log(n)^{1+\gamma_i}}{\alpha(2\beta-1)} n^{-\frac{\beta_i}{2\beta_i+ 1}}\Big)\\
&\geq \exp\Big(-\frac{1}{\alpha(2\beta-1)}\Big).
\end{align*}
Both these results imply Equation~\eqref{eq:boundK}.

\textbf{Proof of Equation~\eqref{separation}:} Consider now $i<j$. From (\ref{eq:pareti2}), each $F_i$ corresponds to the tail index  $\alpha_i = \alpha-t_i = \alpha - \big(\frac{n}{\upsilon \log(M)}\big)^{-\beta_i/(2\beta_i+1)}$. For $i<j$, we have $\alpha_i > \alpha_j$ since $\beta_i > \beta_j$ and since $\frac{n}{\upsilon \log(M)} \geq 1$. Also,  using $\beta_j-\beta_i = (i-j)/M$ with $M \geq \log (n/\log (M))$, and by definition of $\gamma_i, \gamma_j$, 
\begin{align*}
|\alpha_i-\alpha_j| 
&=\Big|t_j(1 - \frac{t_i}{t_j})\Big|\\
&=t_j \left|1-n^{\frac{\beta_j}{(2\beta_j+1)} - \frac{\beta_i}{(2\beta_i+1)} } \big(\log(M)\big)^{\gamma_i - \gamma_j} \right| \\
&= t_j \left|1-\Big(\frac{n}{(\log(M))^{1+\log(\upsilon/\log\log M)}}\Big )^{\frac{\beta_j}{(2\beta_j+1)} - \frac{\beta_i}{(2\beta_i+1)} }\right| \\
&= t_j \left(1-\Big(\frac{n}{\upsilon\log(M) }\Big )^{\frac{(i-j)/M }{(2\beta_j+1)(2\beta_i+1)} }\right) \\
&\geq t_j \exp\left(1-\Big(\frac{(i-j)(M-1)/M }{(2\beta_j+1)(2\beta_i+1)} \Big)\right) \\
&= t_j \left[1- \exp\left( \frac{i-j}{2(2\beta_i+1)(2\beta_j+1)}  \right)\right],
\end{align*}
where the penultimate inequality is obtained since $\upsilon \leq 1$, and since $\log\big(\frac{n}{\log(M)}\big) + 1 \geq M \geq 2$. This implies Equation~\eqref{separation}.

\end{proof}

\begin{proof}[Proof of Lemma~\ref{lem:F0}]
{\bf (1) KL between $F_0$ and $F_i$}

Let $1 \leq i \leq M$.  By definition of KL divergence, 
\begin{align*}
KL(F_0, F_i) &= \int_1^{\infty} f_0(x) \log\left(\frac{f_0(x)}{f_i(x)}\right)dx.
\end{align*}
 
Substituting each densities followed by dividing the integration region, we have 
\begin{align*}
KL(F_0, F_i) &= \int_1^{K_i} \alpha x^{-\alpha - 1} \log(\frac{\alpha x^{-\alpha - 1}}
{\alpha x^{-\alpha - 1}})dx + \int_{K_i}^{\infty} \alpha x^{-\alpha - 1} 
\log\left(\frac{\alpha x^{-\alpha - 1}}{(\alpha - t_i)K_i^{-t_i}x^{-\alpha + t_i - 1}}
\right)dx\\
&=  -\int_{K_i}^{\infty} \alpha x^{-\alpha - 1} \log\left(\frac{\alpha - t_i}
{\alpha} \left(\frac{x}{K_i}\right)^{t_i}\right)dx\\
&=  -t_i \int_{K_i}^{\infty} \alpha x^{-\alpha - 1} 
\log\left(\left(\frac{\alpha - t_i}{\alpha}\right)^{\frac{1}{t_i}} \frac{x}{K_i} \right)dx.
\end{align*}
 By the change of variable $u = \big(\frac{\alpha - t_i}{\alpha}
\big)^{1/t_i} x/K_i$, and letting $a_i = \left(\frac{\alpha-t_i}{\alpha}\right)^{1/t_i}$,

\begin{align*}
KL(F_0, F_i) &=  -t_i \int_{a_i}^{\infty} \alpha \Big(\big(\frac{\alpha}{\alpha - t_i}
\big)^{1/t_i} K_i u\Big)^{-\alpha - 1} \log(u)du \times \Big(\big(\frac{\alpha}
{\alpha - t_i}\big)^{1/t_i} K_i \Big)\\
&= t_i  \Big(a_i^{-1} K_i \Big)^{-\alpha} 
\int_{a_i}^{\infty} (-\alpha) u^{-\alpha - 1} \log(u)du.
\end{align*}
Now by performing an integration by parts, we obtain 
\begin{align*}
KL(F_0, F_i) &=   t_i  \Big( a_i^{-1} K_i\Big)^{-\alpha} \Bigg( \left. u^{-\alpha} 
\log(u)\right|_{a_i}^{\infty}   - \int_{a_i}^{\infty} u^{-\alpha - 1} du \Bigg)\\
&=   t_i  K_i^{-\alpha} \Big( \log(1/a_i) - 
\frac{1}{\alpha} \Big) = K_i^{-\alpha} \left(\log \left(\frac{\alpha}{\alpha-t_i}\right) - \frac{t_i}{\alpha} \right).
\end{align*}
 Using $\alpha-t_i \geq \alpha/2$, we further upper bound this divergence 
\begin{align*}
KL(F_0, F_i) &= K_i^{-\alpha} \left(\log \left(1+\frac{t_i}{\alpha-t_i}\right)-
\frac{t_i}{\alpha}\right)\leq  K_i^{-\alpha} \left(\frac{t_i}{\alpha-t_i} -\frac{t_i}{\alpha} \right) =
 K_i^{-\alpha} \frac{t_i^2}{\alpha(\alpha-t_i)} \\
&= \frac{ 2 t_i^2 K_i^{-\alpha} }{\alpha^2}.
\end{align*}
%

{\bf (2) KL between $F_i$ and $F_0$}

 Similar calculations as above give  
\begin{align*}
KL(F_i,F_0) &= \int_1^\infty f_i(x) \log \frac{f_i(x)}{f_0(x)}dx\\
&= \int_{K_i}^\infty (\alpha-t_i) K_i^{-t_i} x^{-\alpha+t_i-1} 
\log \frac{(\alpha-t_i)K_i^{-t_i}x^{-\alpha+t_i-1}}{\alpha x^{-\alpha-1}} dx\\
&=t_i a_i^{-\alpha+t_i} K_i^{-\alpha} \int_{a_i}^\infty (\alpha-t_i)u^{-\alpha+t_i-1}
 \log (u) du \\
&= K_i^{-\alpha} \left( \log \left(\frac{\alpha-t_i}{\alpha}\right) +
\frac{t_i}{\alpha-t_i}\right).
\end{align*}
Then, the last term can be upper bound in the same way as the case $KL(F_0,F_i)$:
$$
K_i^{-\alpha} \left(\log \left(\frac{\alpha-t_i}{\alpha}\right)+ \frac{t_i}{\alpha-t_i}\right)
 \leq 
\frac{2t_i^2 K_i^{-\alpha}}{\alpha^2}.
$$ 

\end{proof}

\begin{proof}[Proof of Lemma~\ref{lem:Fi}]
{\bf (1) KL between $F_i$ and $F_j$ with $i<j$}

Consider the case $i<j$. First, note that
\begin{align}
 KL(F_i,F_j) &:= \int f_i(x) \log \frac{f_i(x)}{f_j(x)}dx \nonumber\\
&= KL(F_i,F_0) + \int f_i(x) \log \frac{f_0(x)}{f_j(x)}dx \nonumber\\
&= KL(F_i,F_0) + \int_{1}^{K_j} f_i(x) \log \frac{f_0(x)}{f_j(x)}dx + \int_{K_j}^{\infty} f_i(x) \log \frac{f_0(x)}{f_j(x)}dx \nonumber\\
&= KL(F_i,F_0) + \int_{K_j}^{\infty} f_i(x) \log \frac{f_0(x)}{f_j(x)}dx.\label{eq:KLcoucou}
\end{align}
Thus it suffices to bound the second term $\int_{K_j}^{\infty} f_i \log \frac{f_0}{f_j}$ in (\ref{eq:KLcoucou}). 
For any $x \geq K_j$, 
\begin{equation*}
f_0(x)/f_j(x) = (\alpha/(\alpha-t_j)) (K_j/x)^{t_j}
\end{equation*}
is a decreasing function in $x$. Since $\log(\cdot)$ is monotone increasing, the function $\log(f_0/f_j)$ is a decreasing function in $x$ for any $x \geq K_j$.


 For any $x \geq K_j$, we define the conditional distributions $G_0(x)$ and $G_i(x)$ conditioned on the event $\{X>K_j\}$ under distributions $F_0$ and $F_i$ respectively. 
\begin{align*}
G_0(x) &= 1 - \mathbb P_{X \sim F_0}(X \geq x | X \geq K_j) = 1 - \frac{x^{-\alpha}}{K_j^{-\alpha}}, \\
G_i(x) &= 1 - \mathbb P_{X \sim F_i}(X \geq x | X \geq K_j) = 1 - \Big(\frac{x}{K_j}\Big)^{t_i} \frac{x^{-\alpha}}{K_j^{-\alpha}}.
\end{align*}
 By stochastic dominance, we have for any $x \geq K_j$ that $G_0(x) \geq G_i(x)$. This implies that for any \textit{decreasing} function $h$ defined on $[K_j, \infty)$ and also integrable with respect to $G_0$ and $G_i$, we have  
\begin{equation*}
\mathbb E_{G_0} h \geq \mathbb E_{G_i} h.
\end{equation*}
In particular, since $\log(f_0/f_j)$ is a decreasing function in $x$ for any $x \geq K_j$, we have
\begin{equation*}
\mathbb E_{G_0} [\log(f_0/f_j)] \geq \mathbb E_{G_i} [\log(f_0/f_j)],
\end{equation*}
i.e.,~we have
\begin{equation*}
K_j^{\alpha} \int_{K_j}^{\infty} f_0(x)\log \frac{f_0(x)}{f_j(x)}dx \geq \left(\frac{K_i}{K_j}\right)^{t_i} K_j^{\alpha} \int_{K_j}^{\infty} f_i(x) \log \frac{f_0(x)}{f_j(x)} dx.
\end{equation*}
 We use this inequality to bound the second term in (\ref{eq:KLcoucou}). By using Equation~\eqref{eq:boundK}, we get
\begin{align*}
 \int_{K_j}^{\infty} f_i(x) \log \frac{f_0(x)}{f_j(x)} dx &\leq \left(\frac{K_j}{K_i}\right)^{t_i} \int_{K_j}^{\infty} f_0(x)\log \frac{f_0(x)}{f_j(x)} dx= \left(\frac{K_j}{K_i}\right)^{t_i} KL(F_0,F_j)\\
&\leq \exp\left(\frac{1}{\alpha(2\beta - 1)}\right) KL(F_0,F_j).
\end{align*}
Combining this upper bound with bounds on $KL(F_0,F_j)$ and $KL(F_i,F_0)$ in Lemma~\ref{lem:F0} and also with Equation~\eqref{eq:KLcoucou}, 
\begin{align}
 KL(F_i,F_j) &\leq KL(F_i,F_0) + \exp\left(\frac{1}{\alpha(2\beta - 1)}\right) KL(F_0,F_j) \nonumber\\
&\leq \frac{2\exp(\frac{1}{\alpha(2\beta - 1)}) }{\alpha^2}\left( t_j^2K_j^{-\alpha} + t_{i}^2K_{i}^{-\alpha}\right).\label{eq:KLhello}
\end{align}

{\bf (2) KL between $F_i$ and $F_j$ with $i > j$}

 Now we turn to the case $i>j$.  First, note that
\begin{align}\label{eq:Klcoucou2}
 KL(F_i,F_j) &= KL(F_j,F_0) + \int_{K_i}^{\infty} f_j(x) \log \frac{f_0(x)}{f_i(x)}dx.
\end{align}

Again, $\log \frac{f_0(x)}{f_i(x)}$ is a decreasing function for any $x \geq K_i$. Also since $\forall x \geq K_i$, $F_j(x) \leq F_0(x)$, and since $F_j(K_i) = F_0(K_i)$, the measure associated to $F_i$ restricted to $[K_i, \infty)$ stochastically dominates $F_0$. This implies that
\begin{equation*}
 \int_{K_i}^{\infty} f_j(x) \log \frac{f_0(x)}{f_i(x)}dx \leq \int_{K_i}^{\infty} f_0(x)\log \frac{f_0(x)}{f_i(x)} dx = KL(F_0,F_i).
\end{equation*}

This with Equation~\eqref{eq:Klcoucou2} and Lemma~\ref{lem:F0} implies
\begin{align}
 KL(F_i,F_j) &\leq KL(F_j,F_0) + KL(F_0,F_i) \nonumber\\
&\leq \frac{2}{\alpha^2}\left( t_j^2K_j^{-\alpha} + t_{i}^2K_{i}^{-\alpha}\right).\label{eq:KLhello2}
\end{align}


Combining Equations~\eqref{eq:KLhello} and~\eqref{eq:KLhello2}, we obtain the result.
\end{proof}

\noindent{\bf 6. Appendix}
\begin{lem}[Fano's inequality]
Suppose $Y$ is a uniform random variable on $\{1,\ldots, M\}$, and
let $Z$ is a random variable of a function of $X$, where $X|Y=j \sim \pp_j$ with $d\pp_j/d\nu = p_j$ where $\nu$ is the dominating measure.
Then
$$
\pp\left(Z \neq Y\right) \geq 1-\frac{1}{\log M} \left(\frac{1}{M^2}\sum_{j,j'}KL(\pp_{j},\pp_{j'}) + \log 2 \right).
$$
\end{lem}
 
\begin{proof}
Recall the definition of the entropy $H(Y) =- \sum_y p(y) \log p(y)$ for a discrete random variable $Y$ with a probability mass function $p(y)$.
Also we denote $H(Y|Z=z)$ by the conditional entropy of $Y$ given $Z=z$, and we define $H(Y|Z) = -\sum_x \sum_y p(y,z) \log p(y|z)$. 
Following the terminology used in the information theory, we define \textit{information} between $Y$ and $Z$ as the KL divergence between joint distribution and product of the marginal distribution, i.e.
$I(Y,Z) = KL(P_{Y,Z}, P_Y \times P_Z)$ where we can show that 
\begin{equation}\label{info}
 I(Y,Z) =  KL(P_{Y,Z}, P_Y \times P_Z) = H(Y)-H(Y|Z)
\end{equation}
 by splitting the probability distribution.
Finally recall that for $Z= Z(X)$, $I(Y,Z) \leq I(Y,X)$.

Consider the event $E = \mathbf{1} \{ Z \neq Y\}$.
By splitting the probabilities with different order, 
\begin{align*}
H(E,Y|Z) &= H(Y|Z) + H(E|Y,Z) := (1) \\
&= H(E|Z)+H(Y|E,Z) :=(2),
\end{align*}
where $(1) = H(Y|Z)$ since $E$ becomes a constant given $Y$ and $Z$. Then we upper bound (2) as follows,
\begin{align*}
(2) &=  H(E|Z)+H(Y|E,Z) \\
&\leq H(E) + H(Y|E,Z)\\
&= H(E) + \pp(E=0)H(Y|E=0,Y) + \pp(E=1)H(Y|E=1,Z)\\
&\leq \log 2 + \pp(Z \neq Y) \log M.
\end{align*}
Combining both (1) and (2), we have 
$$
H(Y|Z) \leq \log 2 + \pp(Z \neq Y)\log M,
$$
in turn, 
\begin{equation}\label{fano}
\pp(Z \neq Y) \geq \frac{1}{\log M} \left( H(Y|Z) - \log 2\right).
\end{equation}
Now, using the fact (\ref{info}),
\begin{align}
H(Y|Z) &= \log M - I(Y,Z) \nonumber\\
&\geq \log M - I(Y,X)\nonumber\\
&= \log M - \int \sum_y p(y) p(x|y) \log \frac{p(y) p(x|y)}{p(x) p(y)}\nonumber \\
&= \log M - \int \sum_j \frac{1}{M} \mathbb 1 \{ y=j \} p(x|y) \log \frac{p(x|y)}{p(x)}\nonumber \\
&= \log M - \frac{1}{M} \sum_{j=1}^M \int p_j(x) \log \frac{p_j(x)}{ \frac{1}{M} \sum_{j'} p_{j'}(x)} dx\nonumber\\
&\geq \log M - \frac{1}{M^2} \sum_{j , j'} KL(\pp_j, \pp_{j'}), \label{fano2}
\end{align}
where the penultimate equality is followed since $p(x) = \sum_j \pp(Y=j) \pp(X=x|Y=j) = \frac{1}{M} \sum_j p_j(x)$, and the last inequality is obtained by the concavity of the logarithm function.
Combining (\ref{fano}) and (\ref{fano2}), we obtain 
$$
\pp(Z \neq Y) \geq 1- \frac{1}{\log M} \left( \frac{1}{M^2} \sum_{j,j'} KL(\pp_j, \pp_{j'}) + \log 2 \right).
$$
\end{proof}


\noindent {\large\bf Acknowledgment}

The authors are grateful to Richard J. Samworth and Richard Nickl for their comments and advice.
\par

\vspace{10pt}
\noindent{\large\bf References}
\begin{description}

\bibitem[Beirlant, et. al.(1996)]{beirlant1996}
Beirlant, J., and Vynckier, P., and Teugels, J. (1996). Tail index estimation, Pareto quantile plots and regression. \newblock \emph{Journal of American Statistical Association}, \textbf{70}, 1659--1667.
 
\bibitem[Cover and Thomas(2012)]{coverbook}
Cover, T. M. and Thomas, J. A. (2012). \newblock \emph{Elements of Information Theory}. Wiley-interscience.

\bibitem[Danielsson, et. al.(2001)]{danielsson2001}
Danielsson, J. and de Haan, L. and Peng, L. and de Vries, C.G. (2001).
\newblock Using a Bootstrap Method to Choose the Sample Fraction in Tail Index Estimation.
\newblock \emph{Journal of Multivariate Analysis}, \textbf{2}, 226--248

\bibitem[Drees(2001)]{drees2001}
Drees, H. (2001). Minimax risk bounds in extreme value theory.
\newblock \emph{The Annals of Statistics}, \textbf{29}, (1) 266--294.

 \bibitem[Drees and Kaufmann(1998)]{drees1998}
Drees, H. and Kaufmann, E. (1998). Selecting the optimal sample fraction in univariate extreme value estimation.
\newblock \emph{Stochastic Processes and Their Applications}, \textbf{75}, 149--172.

\bibitem[Gin{\'e} and Nickl(2010)]{gine2010confidence}
Gin{\'e}, E. and Nickl, R. (2010). Confidence bands in density estimation.
\newblock \emph{The Annals of Statistics}, \textbf{38}, (2) 1122--1170.

\bibitem[Gomes, et. al.(2012)]{gomes2012}
Gomes, M. I. and F. Figueiredo, and Neves, M. (2012). Adaptive estimation of heavy right tails: resampling-based methods in action. \newblock \emph{Extremes}, \textbf{15}. 463--489

\bibitem[Gomes, et. al.(2008)]{ivette2008tail}
Gomes, I. M and De Haan, L., and Rodrigues, L{\'\i}gia Henriques (2008). Tail index estimation for heavy-tailed models: accommodation of bias in weighted log-excesses. \newblock \emph{Journal of the Royal Statistical Society: Series B}, \textbf{91}, 31--52.

\bibitem[Grama and Spokoiny(2008)]{spokoiny}
Grama, I. and Spokoiny, V. (2008). Statistics of extremes by oracle estimation.
\newblock \emph{The Annals of Statistics}, \textbf{36}, (4) 1619--1648.

\bibitem[de Haan and Ferreira(2006)]{dehaan2006}
de Haan, L. and Ferreira, A. (2006). \emph{Extreme Value Theory: An Introduction}. Springer series in operations research und financial engineering.

\bibitem[Hall(1982)]{hall1982}
Hall, P. (1982). On some simple estimates of an exponent of regular variation.
\newblock \emph{Journal of the Royal Statistical Society: Series B.}, \textbf{44}, (1) 37--42.

\bibitem[Hall and Welsh(1984)]{hall1984}
Hall. P. and Welsh, A. H. (1984). Best attainable rates of convergence for estimates of parameters of regular variation.
\newblock \emph{The Annals of Statistics}, \textbf{12}, (3) 1079--1084.

\bibitem[Hall and Welsh(1985)]{hall1985adaptive}
Hall. P. and Welsh, A. H. (1985). Adaptive estimates of parameters of regular variation.
\newblock \emph{The Annals of Statistics}, \textbf{75}, (1) 331--341.

\bibitem[Hill(1975)]{hill1975simple}
Hill, B. M. (1975). A simple general approach to inference about the tail of a distribution.
\newblock \emph{The Annals of Statistics}, \textbf{3}, (5) 1163--1174.

\bibitem[Lepski(1992)]{lepski1992problems}
Lepski, O. V. (1992).
\newblock On problems of adaptive estimation in white gaussian noise.
\newblock \emph{Topics in nonparametric estimation}, \textbf{12}, 87--106.

\bibitem[Novak(2013)]{novak2013}
Novak, S. Y. (2013). Lower bounds to the accuracy of inference on heavy tails.
\newblock \emph{Bernoulli}(to appear)

\bibitem[Pickands(1975)]{pickands1975statistical}
Pickands, J. (1975). Statistical inference using extreme order statistics.
\newblock \emph{The Annals of Statistics}, 119--131.

\bibitem[Spokoiny(1996)]{spok}
Spokoiny, V. G. (1996). Adaptive hypothesis testing using wavelets.
\newblock \emph{The Annals of Statistics}, \textbf{24}, (6) 2477--2498.

\bibitem[Van der Vaart(2000)]{vaartbook}
Van der Vaart, A. W.(2000). \emph{Asymptotic Statistics}. Cambridge University Press.

\end{description}


\vskip .65cm
\noindent
University of Cambridge
\vskip 2pt
\noindent
E-mail: a.carpentier@statslab.cam.ac.uk
\vskip 2pt

\noindent
University of Cambridge
\vskip 2pt
\noindent
E-mail: a.kim@statslab.cam.ac.uk
\vskip .3cm


\end{document}